\documentclass[11pt]{amsart}
\usepackage{amsmath, amsthm, amssymb}
\usepackage{verbatim}
\usepackage{enumitem}
\usepackage{xcolor}
\allowdisplaybreaks
\usepackage{url}

\newcommand{\dbar}{\overline{\partial}}
\newcommand{\ddbar}{\sqrt{-1}\partial\overline{\partial}}

\theoremstyle{plain}
\newtheorem{thm}{Theorem}[section]
\newtheorem{prop}[thm]{Proposition}
\newtheorem{lem}[thm]{Lemma}
\newtheorem{cor}[thm]{Corollary}

\newtheorem{ex}[thm]{Example}

\newtheorem{theorem}{Theorem}

\theoremstyle{definition}

\newtheorem{rem}[thm]{Remark}

\renewcommand{\[}{\begin{equation}}
\renewcommand{\]}{\end{equation}}

\newcommand{\al}{\alpha}
\newcommand{\be}{\beta}
\newcommand{\ga}{\gamma}
\newcommand{\la}{\lambda}
\newcommand{\ka}{\kappa}
\newcommand{\ep}{\epsilon}

\newcommand{\Ga}{\Gamma}

\newcommand{\vep}{\varepsilon}

\newcommand{\n}{\Vert}
\newcommand{\Ar}{\mathrm{Area}}
\newcommand{\Lip}{\mathrm{Lip}}

\newcommand{\NN}{\mathbb{N}}
\newcommand{\ZZ}{\mathbb{Z}}

\newcommand{\RR}{\mathbb{R}}
\newcommand{\CC}{\mathbb{C}}

\newcommand{\PP}{\mathbb{P}}

\newcommand{\Rm}{\mathrm{Rm}}
\newcommand{\Rc}{\mathrm{Rc}}

\newcommand{\mult}{\mathrm{mult}}

\newcommand{\sO}{\mathcal{O}}

\newcommand{\PSH}{\mathrm{PSH}}

\numberwithin{equation}{section}

\usepackage{hyperref}
\hypersetup{colorlinks,citecolor=blue,plainpages=false,hypertexnames=false}

\title[Uniformisation of  K\"ahler surfaces with $sec >0$]{Uniformisation of complete K\"ahler surfaces with positive sectional curvature}
\author[V. Datar]{Ved Datar}
\author[V. P. Pingali]{Vamsi Pritham Pingali}
\author[H. Seshadri]{Harish Seshadri}

\address{Department of Mathematics, Indian Institute of Science, Bengaluru, India}
\email{vvdatar@iisc.ac.in}
\email{vamsipingali@iisc.ac.in}
\email{harish@iisc.ac.in}
\thanks{The first author (Datar) is supported in part by ANRF MATRICS grant MTR/2022/000260 and an INSA Young Associate fellowship.}

\begin{document}

\maketitle

\begin{abstract}
We prove that any complete non-compact K\"ahler surface with positive sectional curvature is biholomorphic to $\CC^2$, establishing the two dimensional case of the weaker form of Yau's uniformisation conjecture. In contrast to all previous results, no assumptions are made on the geometry at infinity.

The proof introduces a new approach towards Yau-type uniformisation problems, based on uniformly Lipschitz plurisubharmonic weight functions with finite Monge-Amp\`ere mass, and weighted $L^p$ holomorphic functions. A central difficulty is that these weights are neither smooth nor proper. As a consequence of the method, we also obtain B\'ezout-type intersection and multiplicity estimates in considerable generality. In a different direction, we also prove a new obstruction to the existence of complete K\"ahler metrics with non-negative bisectional curvature on non-compact K\"ahler manifolds, and use it to construct new examples admitting no such metrics. We conclude by discussing possible extensions of our methods to higher dimensions and related open problems.\end{abstract}

\tableofcontents
\section{Introduction}

A recurring theme in K\"ahler geometry is that curvature positivity imposes strong holomorphic rigidity. In the compact setting, this principle is strikingly illustrated by the Frankel conjecture, resolved independently by Mori \cite{Mo} and Siu-Yau \cite{SY}. To state their result, recall that a K\"ahler manifold $(X,\omega)$ is said to have positive bisectional curvature (which we denote by  $BK>0$) if for all real unit tangent vectors $u,v\in TX$, we have $$BK(u,v):= \Rm(u,v,v,u) + \Rm(u,Jv,Jv,u) > 0,$$ where $J$ is the complex structure and $\Rm$ is the Riemann curvature tensor of the induced Riemannian metric $g(\cdot,\cdot) = \omega(\cdot,J\cdot)$. The Siu-Yau-Mori theorem asserts that a compact K\"ahler $n$-manifold with $BK>0$ must be biholomorphic to the complex projective space of dimension $n$. In contrast, the non-compact case remains far more mysterious. A deep conjecture of Yau \cite{Y} predicts that any complete non-compact K\"ahler manifold with $BK>0$ must be biholomorphic to $\mathbb{C}^n$.  A weaker conjecture, due to Green and Wu \cite{GW}, is that the same conclusion holds under the stronger hypothesis of positive sectional curvature. While the soul theorem of Gromoll--Meyer shows that such manifolds are necessarily diffeomorphic to $\mathbb{R}^{2n}$, the problem of determining their global complex structure has remained wide open, except in dimension one.  Even for Riemann surfaces, Yau's conjecture is non-trivial and follows from the classical uniformization theorem in combination with the results of Cohn-Vossen \cite{CV} and Huber \cite{Hub}. It is worth emphasizing that $\CC^n$ itself does admit many (in fact rotationally symmetric) complete K\"ahler metrics with positive bisectional, or even sectional curvature. The first such example, with $BK>0$, appears to be due to Klembeck \cite{K}. Subsequently, Cao discovered two gradient shrinking solitons \cite{Cao1,Cao2} on $\CC^n$ with positive sectional curvatures, and Wu-Zheng in \cite{WZ} obtained a complete classification of rotationally symmetric complete K\"ahler metrics on $\CC^n$ with positive bisectional or sectional curvature.

The Yau conjecture has been settled under additional hypotheses, typically involving volume growth, curvature decay and finiteness of certain curvature integrals starting with the pioneering works of Mok (cf. \cite{M,Mok89,Mok90,MZ}). The sharpest result in this direction, obtained independently by Liu \cite{L4} and Lee-Tam \cite{LT} using very different methods, states that if $(X,\omega)$ has $BK>0$ and has maximal volume growth in the sense that there exists a constant $\ka>0$ and a point $o\in X$ such that $$|B(o,r)| \geq \ka r^{2n}$$ for all $r>0$, then $X$ is biholomorphic to $\CC^n$. Here, for any subset $A\subset X$ we denote its volume by $|A|$. Lee and Tam proved this result using the K\"ahler-Ricci flow, building on earlier work of Chau-Tam \cite{CT}; the maximal volume growth plays a key role in the proof of existence of the K\"ahler-Ricci flow for a definite time. In contrast, Liu's method is closer in spirit to Mok's original approach. A key step in Liu's proof is the embedding of $X$ into the complex Euclidean space by using polynomial growth holomorphic functions. The limitations of this strategy are highlighted by another striking result of Liu, which shows that the space of polynomial growth holomorphic functions on $X$ is nontrivial if and only if $(X,\omega)$ has maximal volume growth. In particular, this precludes any direct extension of Liu’s method to the non-maximal growth setting.

 Interestingly, in \cite{CZ2}, Chen-Zhu proved that if $X$ is a complete non-compact K\"ahler manifold with $BK>0$ and $o \in X$, then there exists $\ka>0$ such that $$|B(o,r)| \ge \ka r^{n},$$ where $n =dim_{\CC}X$, and it is well known this bound is sharp. Indeed Klembeck's example \cite{K}, as well as one of Cao's solitons \cite{Cao1}, have this minimal volume growth at infinity. In  \cite{CZ3}, Chen-Zhu proved that if $X$ has positive {\it sectional} curvature and minimal volume growth, then it is biholomorphic to an affine algebraic variety. If $dim_{\CC}X=2$, then it is biholomorphic to $\CC^2$ thanks to a result of Ramanujam \cite{R}. Moreover, in \cite{CZ2} it was proven that if $$\int_X Ric^n < \infty,$$ then $X$ is biholomorphic to a quasi-projective variety. In a recent work \cite{DPS}, we proved that if $X$ is a surface with positive  sectional curvature, then the above integral bound on the square of the Ricci form always holds. If, in addition, the sectional curvature is bounded above,  $X$ is biholomorphic to $\CC^2$ by \cite{CZ2,R}.

The main contribution of the present work is to remove the requirement of an upper bound on the sectional curvature, thereby completely  resolving the weak form of Yau's conjecture for K\"ahler surfaces. Namely, we prove the following:
\begin{theorem} \label{thm:main}
Let $(X, \omega)$ be a complete non-compact K\"ahler surface with positive sectional curvature. Then $X$ is biholomorphic to $\CC^2$. \end{theorem}
As remarked above, if $X$ has positive sectional curvature, then $X$ is diffeomorphic to $\RR^4$ by  the Gromoll-Meyer theorem \cite{GM}. Moreover, by the work of Green-Wu \cite{GW}, it also follows that $X$ is strongly Stein in the sense that there exists a smooth plurisubharmonic exhaustion function $\rho:X\rightarrow\RR$ with $|\nabla \rho| \leq 1$.  It is important to note that neither of these results is known to hold under the weaker assumption of $BK>0$.  Indeed, little is known if one replaces positive sectional curvature with positive holomorphic bisectional curvature (without additional global hypotheses such as volume growth, curvature decay, etc). For instance, it appears to be unknown if the complex surface $\CC \times \CC^\ast$ admits a complete K\"ahler metric with $BK>0$. In the case of $BK>0$, the proof of our main theorem actually shows the following stronger result.

\begin{theorem}\label{thm:main-strong}
Let $(X,\omega)$ be a complete non-compact K\"ahler surface with $BK>0$ and satisfying the following two additional conditions:
\begin{enumerate}
\item $X$ is strongly Stein in the above sense.
\item  $X$ is contractible.
\end{enumerate}
Then $X$ is quasi-projective. In particular, if $X$ is also simply-connected at infinity, by Ramanujam's criterion \cite{R}, it follows that $X$ is biholomorphic to $\CC^2$.
\end{theorem}

 \subsection{An outline of the proof} As noted above, if $(X,\omega)$ is a complete, non-compact K\"ahler surface with positive sectional curvature, then $X$ is diffeomorphic to $\RR^4$. Moreover, by a classical criterion of Ramanujam \cite{R}, the proof of the main theorem then reduces to proving that $X$ is quasi-projective. Our basic strategy in proving quasi-projectivity is motivated by the strategy employed by Chen-Zhu in \cite{CZ2}, which itself is heavily inspired by several papers of Mok (cf. \cite{Mok90}  and references therein). There are however several significant differences, and additional technical challenges that need to be overcome.

Recall that in \cite{CZ2}, the projective embedding is produced using sections of the anti-canonical bundle $K_X^{-1}$. Inspired by an idea of Demailly (cf.\ \cite{D}), we instead work with weighted $L^p$ integrable functions,  which we view as sections of tensor powers of the trivial line bundle.  Our starting point is the Lipschitz-continuous plurisubharmonic weight function $\phi$ constructed in \cite{DPS} with finite Monge-Amp\`ere mass. We then consider the tensor powers $L^q$ of the trivial line bundle $L$, endowed with the Hermitian metrics $h_{q\phi} = e^{-q\phi}$  and $h_{q\psi}  = e^{-q\psi}$, where $\psi = P_1\phi$, the time one heat flow smoothing of $\phi$. By a well known result of Ni-Tam \cite{NT}, $\psi$ is also plurisubharmonic. We denote the corresponding norms by $\n s\n_{q\phi}$ and $\n s\n_{q\psi}$ respectively. We now let $R_q$ denote the subspace of $H^0(X,L^q)$ generated by all sections $s\in H^0(X,L^q)$  satisfying $$\int_X |s|^{\frac{2}{k}} e^{-q\psi/k}\omega^2 < \infty \text{ for some $k \in \NN$},$$  and we set  $R = \oplus_{q\in \NN}R_q.$ By H\"ormander-type $L^2$ estimates (applied to $k=1$), the space $R_q$ is non-empty if $q>>1$.  Next, by an elementary application of H\"older's inequality one can show that $R_{q}\cdot R_{q'} \subset R_{q+q'}$ and hence $R$ is a graded algebra. We denote by $M$  the function field of $R$ ie.\ the degree zero piece of the fraction field of $R$.

Before continuing with the proof outline, we point out the key new technical tool of this paper: B\'ezout estimates (cf. Proposition \ref{prop:Bezout} and Proposition \ref{prop:Bezout-proj}) based on the uniformly Lipschitz weight $\phi$. While these estimates are motivated by analogous estimates in \cite{CZ2} and many of Mok's papers (cf. \cite{MZ}), and crucially depend on the finiteness of the Monge-Amp\`ere mass, there are several new technical difficulties that crop up in our formulation. The B\'ezout estimates are used at several places to establish various finiteness results. We only mention two important consequences here. First, a direct consequence of Proposition \ref{prop:Bezout} is that there exists a constant $C>0$ such that for any $s\in R_q$ and any $q,\ep>0$ we have $$\int_X \Big(\ddbar\log(|s|^2+ \ep^2) + q\ddbar\phi\Big)\wedge Ric \leq Cq.$$ Using the Poincar\'e-Lelong formula, the integral estimate above implies that the vanishing order of any $s\in R_q$ at a given point $o$ is bounded above by $Cq$ for some constant $C$ independent of $q$ and $s$.  Following Mok, one can then prove that  $M$ is finitely generated and the spaces $R_q$ are finite dimensional.

 Once these are established, the proof of Theorem \ref{thm:main} proceeds in two main stages.
 \begin{itemize}
 \item (an almost embedding) We first construct a birational embedding $F:X\dashrightarrow \PP^N$ of the form $$F(x) = [s_0:s_1:\cdots:s_N],$$ where each $s_i$ lies in a fixed $R_q$ and $F$ is a biholomorphism from $U = X\setminus V$ onto its image, where $V$ is an analytic set. Next, we prove that $F(U)$ is Zariski open in a projective surface $Z\subset \PP^N$ using a criterion due to Simha \cite{Simha}. A key point is that $V$ is cut-out by holomorphic functions, and hence $U$ is Stein.
 \item (a full embedding) By using the B\'ezout estimates again, we then prove that by increasing $q$, we obtain a well defined Kodaira map: $\Phi:X\rightarrow \PP^N$ using sections of $R_q$ ie. $$\Phi(x) = [s_0(x):\cdots :s_N(x)],$$ where we abuse notation and label $\dim R_q - 1$ by $N$ again.  Furthermore, by choosing a $q$ large enough we can also ensure that $\Phi$ is a global immersion and that $\Phi(X)$ is contained in  a projective surface $Z\subset\PP^N$. One would now like to prove injectivity of $\Phi$ by following the argument in the proof of Proposition 3.3 in \cite{Mok90}. The argument there however does not directly apply since $\Phi(X)$ may intersect the singular locus of $Z$ and it may happen that two disjoint neighbourhoods in $X$ may map to separate branches that intersect at a single point. This is because $Z$ may not be locally irreducible. Our final trick is to consider the normalization $\tilde Z$ of $Z$ and lift the map $\Phi$ to a map $\tilde\Phi : X\rightarrow \tilde Z \subset \PP^{\tilde N}$ which is still an immersion. The argument in \cite{Mok90} then implies that $\tilde\Phi$ is injective and hence an embedding. The final step is to prove quasi-surjectivity by once again using Simha's criterion. The fact that Simha's criterion can still be applied to a map to the normalization relies on an elementary algebro-geometric observation due to ChatGPT pro 5.2 (cf.\ Lemma \ref{lem:algebro-geometric}). Applying Simha's criterion is slightly delicate in our work and we clarify some of the terse arguments in \cite{Mok89}.
 \end{itemize}

 We now make some remarks on the B\'ezout estimates. First, apart from their applications mentioned above, it is used to show the finiteness of the Gauss-Bonnet integrals of Riemann surfaces cut out by sections in $R_q$. This estimate plays a key role in the proof of quasi-surjectivity using a criterion of Simha \cite{Simha} in combination with some ideas from Mok's works.  Our next remark is about the proof of the B\'ezout estimates: The proof involves a careful integration by parts argument, and the use of Bedford-Taylor theory (on account of a lack of regularity of $\phi$). An added technical complication is that the sections in $R_q$ are sums of sections which lie in different $L^p$ spaces. A key ingredient is that if $s\in H^0(X,L^q)$ is $L^p$-integrable, then $\n s\n_{q\phi}$ as well as $\n s^Nds\n_{q\phi}$ (for $N+1>p$) are of polynomial growth. This is proved first for the norms measured with respect to the weight $q\psi$, using the Bochner inequality combined with the upper bound $\ddbar\psi \leq A\omega$, and a mean value inequality. To go back to the weight $q\phi$ we then rely on the basic estimate (cf. Lemma \ref{lem:heat-est}) that $e^{-q\phi}\lesssim e^{-q\psi}$. Note that by the aforementioned result of Liu \cite{L2}, if $(X,\omega)$ has positive bisectional curvature, there are no non-trivial holomorphic functions with polynomial growth unless $(X,\omega)$ has maximal volume growth. In contrast, once we measure the norms with respect to the weight $q\phi$, we automatically have plenty of polynomial growth holomorphic functions.
\subsection{Comparison to previous works} We now compare our approach with some important previous works.
 \begin{itemize}
 \item {\bf Comparison with the work of Chen-Zhu \cite{CZ2}:} As noted above, Chen-Zhu proved (inspired heavily by Mok \cite{Mok89, Mok90, MZ}) that positive and bounded sectional curvature, along with an integral bound on $Ric^2$  implies that $X\cong \mathbb{C}^2$. In a previous work \cite{DPS}, we used this criterion to prove Yau's conjecture under the assumption of positive and bounded sectional curvature. The upper bound on the sectional curvature was used by Chen-Zhu in a few key steps:
\begin{enumerate}
\item An injectivity radius bound implies that $L^2$ sections of $-qK_X$ form an algebra, because they are bounded. In fact, it was proven that even their gradient was bounded.
\item The image under a Kodaira-type map to $\mathbb{P}^N$ is locally Stein (a lemma due to Mok).
\item Several intersection numbers are finite thanks to Mok's B\'ezout estimates.
\end{enumerate}
In our case, we choose to embed using sections of a trivial bundle equipped with a heat-kernel smoothing of a solution to a Monge-Amp\`ere equation. It is non-trivial (cf. Lemma \ref{lem:heat-est}, which is the same as Lemma 6 in \cite{DPS}, but included here with proof for the convenience of the reader) to compare the smoothing with the solution itself. The local Steinness of the image follows from the fact that complements of zero loci of holomorphic functions on Stein manifolds are Stein. The problem with $L^2$ sections not forming an algebra is remedied by considering $L^p$ sections as we explain below. The key point is that the B\'ezout estimates are technically challenging and not a routine adaptation of Mok's estimates.
 \item {\bf Comparison with Demailly's work \cite{D}:} Demailly proved a purely function-theoretic characterisation of affine varieties in \cite{D}. He got around the problem of $L^2$ sections being unbounded by considering $L^p$ sections (with a slightly unusual $L^p$ norm). However, \cite{D} assumed that the weight function was an exhaustion.  This strong assumption was used to obtain natural inclusions between $L^p$ spaces. As a consequence, taking inspiration from polynomials for nomenclature, Demailly's (like Chen-Zhu and Mok) spaces of sections are indexed by only one number - the ``degree". We cannot prove (yet) that the solution to the Monge-Amp\`ere equation is indeed an exhaustion. Therefore a priori, we would have had a bidegree - $(p,q)$ corresponding to $L^p$ sections of the trivial bundle with $e^{-q\psi}$ as the weight. This would have been disastrous towards the end of the proof where Mok's Segre and Veronese embedding tricks are used to bring the degrees (of embeddings on various open sets) to the same level. We employ a novel idea of considering \emph{finite sums} of various $L^p$ sections (fixing $q$ as the degree). As pointed out above, this in turn makes the B\'ezout estimates even more complicated than they would have been otherwise.
 \end{itemize}

 \subsection{Organisation of the paper} In Section \ref{sec:background} we collect some well known results and prove some preliminary estimates that will be fundamental to what follows. In Section \ref{sec:Bezout} we prove our B\'ezout estimates. This section is the technical heart of the paper. In Section \ref{sec:Bezout-consequences} we derive two consequences of the B\'ezout estimates. The first is a multiplicity estimate that plays a key role in proving the finite generation of $M$ and the finite-dimensionality of $R_q$. The second consequence that we prove is an estimate on the Gauss-Bonnet integral of a Riemann surface arising as a zero locus of a section. We then use these estimates to complete the proof of the main theorem in Section \ref{sec:proof}. Finally, in Section \ref{sec:discussion} we discuss some open questions and partial results. In particular, we prove a new obstruction to the existence of complete K\"ahler metrics with non-negative bisectional curvature on non-compact K\"ahler manifolds, and provide new examples of manifolds that admit no such metrics.

\section{Background and basic results}\label{sec:background}

We collect some basic results in this section that will be used throughout the rest of the paper.

\subsection{Existence of $L^2$-integrable holomorphic sections} As a first step in our proof of the main theorem, we will consider the Kodaira map of $X$ into $\CC\PP^N$ using sections of a holomorphic line bundle. We recall the following well known result of Andreotti-Vesentini to construct holomorphic sections based on the H\"ormander technique.

\begin{thm}[H\"ormander, Andreotti-Vesentini]\label{thm:l2-existence} Let $(X,\omega)$ be a complete K\"ahler manifold, let $u$ be a smooth function on $X$, and let $L$ be a holomorphic line bundle equipped with a smooth Hermitian metric $h$ such that the curvature satisfies $$\ddbar u + \sqrt{-1}\Theta_h + Ric \geq c(x)\omega,$$ for some continuous function $c:X\rightarrow (0,\infty)$. Suppose we have an $L$-valued $(0,1)$ form $\be$ satisfying $\dbar \be = 0$ and $$\int_X \frac{||\be||^2}{c}e^{-u} \omega^n < \infty.$$ Then there exists a $\xi \in \Ga(L)$ satisfying $\dbar \xi = \be$ and the $L^2$-estimate $$\int_X\frac{|\xi|^2}{c}e^{-u}\omega^n \leq \int_X \frac{||\be||^2}{c}e^{-u} \omega^n.$$
\end{thm}

Note that $\xi$ is uniquely determined if an additional orthogonality condition is imposed. In applications one requires the above theorem for possibly singular hermitian metrics $h$. This is standard, and is achieved by a smoothing argument. We need the following basic consequence.

\begin{prop}\label{prop:section-existence-hormander}
Let $(X^n,\omega)$ be a complete non-compact K\"ahler manifold with $Ric\geq 0$. Let $\psi \in C^\infty(X)\cap \PSH(X)$ that is strictly $\PSH$ in the sense that $$\ddbar\psi(x) \geq c(x)\omega$$ for some continuous function $c:X\rightarrow (0,\infty)$. Then
\begin{enumerate}
\item There exists a $q_0$ such that for any $q \geq q_0$ there exists a non-trivial section $\sigma\in H^0(X,K_X)$ such that $$\int_X \sigma \wedge \bar\sigma~ e^{-q\psi} < \infty. $$
\item Given any two points $x_1,x_2\in X$, there exists a $q_0>0$ such that for any $q>q_0$, there exists a holomorphic function $s$ on $X$ such that $s(x_1) = 0$ but $s(x_2) \neq 0.$  Moreover, $$\int_X |s|^2e^{-q\psi}<\infty.$$
\item Fix a point $o\in X$. Then there exists a $q_0 > 0$ such that for any $q \geq q_0$,  there exist $n$ holomorphic functions $s_1,\cdots,s_n$ on $X$ such that $s_1(o) = \cdots = s_n(o) = 0$ and $$ds_1 \wedge\cdots\wedge ds_n(o) \neq 0.$$In particular $\{s_1,\cdots,s_n\}$ form coordinates near $o$. Moreover, each $s_i$ satisfies $$\int_X |s_i|^2e^{-q\psi}<\infty.$$

\end{enumerate}

\end{prop}

\begin{proof}
The proofs are standard, but we include them here for completeness. Fix a point $o\in X$. By scaling the metric we may assume that we have holomorphic coordinates $(z^1,\cdots,z^n)$ on the ball $B(o,2)$. Let $\chi \in C^\infty(X)$ be a cut-off function such that $0\leq \chi\leq 1$, $\chi\equiv 1$ on $B(o,1)$ and the support of $\chi$ is contained in $B(o,2)$.
\begin{enumerate}
\item Let $\tau = \chi dz^1\wedge\cdots\wedge dz^n$ and $\be = \dbar_{K_X} \tau$, where $\dbar$ operator is on the line bundle $K_X$. In other words, locally $$\be = (\dbar \chi) dz^1\wedge\cdots\wedge dz^n.$$ We now apply the above theorem with the Hermitian metric $(\omega^n)^{-1}$ and the weight $$\tilde\psi = q\psi + (n+1)\chi \log|z|^2.$$ Then it is easy to see that there exists a function $\tilde c \in C^0(X)$ such that $\tilde c>0$ and $\ddbar\tilde\psi \geq \tilde c.$ Since $\sqrt{-1}\Theta_h = -Ric$ in this case, we have that $$\ddbar\tilde \psi + \sqrt{-1}\Theta_h + Ric \geq \tilde c\omega.$$ Moreover, since $\be$ is compactly supported, $$\int_X \frac{\n \be\n^2}{\tilde c}e^{-\tilde\psi}\omega^n < \infty.$$ Then by the theorem above there exists a section $\xi\in \Ga(K_X)$ such that $\dbar_{K_X}\xi = \be$ and $\xi$ satisfies $$\int_X \frac{\n \xi\n^2}{\tilde c} e^{-q\psi - (n+1)\chi\log|z|}\omega^n < \int_X \frac{\n \be\n^2}{\tilde c}e^{-\tilde\psi}\omega^n<\infty.$$ Clearly $\xi(o) = 0$. We now let $\sigma = \tau - \xi.$ Clearly $\sigma \in H^0(X,K_X)$ Next, since $\tau(o) \neq 0$, $\sigma(o) \neq 0$ and hence $\sigma$ is a non-trivial section. Finally one can choose $\tilde c\leq 1$ and so by the compactness of the support of $\chi$ one can also easily see that $$\int_X \sigma\wedge\bar\sigma ~e^{-q\psi} < \infty.$$
\item Let $z_j:= (z^1_j,\cdots,z_j^n)$ be holomorphic coordinates near $x_j$. We let $\chi_j$ be a cut-off function as above which is identically one near $x_j$. Moreover $x_1 \notin\mathrm{Supp}(\chi_2)$ and vice versa. We then let $\be = \dbar\chi_2$ and consider the weight $$\tilde\psi = q\psi + (n+1)\chi_1\log|z_1|^2 + (n+1)\chi_2\log|z_2|^2.$$ Once again, by compactness of the supports of the cut-off function, there exists a $\tilde c >0$ such that $$\ddbar\tilde\psi + Ric \geq \tilde c\omega.$$ By arguing as above, there exists a function $\xi$ such that $\dbar\xi = \be$. Moreover, by the choice of our weights, $\xi(x_1) = \xi(x_2) = 0$. Then $s = \chi_2-\xi$ is the required holomorphic function.
\item Let $t_i = z^i\chi$, and let $\be_i = \dbar t_i.$ Note that $dt_i (o) = dz^i (o).$ We now let $L$ be the trivial bundle endowed with the trivial Hermitian metric $h$. Then, since the Ricci curvature is non-negative we still have $$\ddbar\tilde \psi + \sqrt{-1}\Theta_h + Ric \geq \tilde c\omega.$$ The same argument as above then shows that we have a smooth functions $\{\xi_i\}_{i=1}^n$ such that $\dbar\xi_i = \be_i$, $\xi_i(o) =0$ and $d\xi_i(o) = 0$. We now let $s_i = t_i - \xi_i$. Then $s_i(o) = 0$ and $ds_i(o) = dz^i$, and this completes the proof.
\end{enumerate}
\end{proof}

\subsection{Currents, distributional inequalities and intersection numbers}

Let $L$ be a Hermitian line bundle with a smooth Hermitian metric $h$. We consider hermitian metrics $h_u= e^{-u}h$ where $u$ is a Lipschitz function. We now fix the following notation: For $s\in H^0(X,L)$ and $\vep,p>0$ we set $$\zeta_{\ep}(s,p,u,h):= \ddbar\log(\n s\n_{h_u}^p+ \ep^2) + \frac{p}{2}\ddbar u,$$ where the norm is defined by $$\n s\n_{h_u}^2:= s\bar s e^{-u}h. $$ On some occasions we will have use for the following modification: $$\tilde{\zeta}_{\ep}(s,p,u,h):= \ddbar\log(\n s\n_{h_u}^p+ \ep^2) + \frac{p}{2}\ddbar u + \frac{p}{2}\sqrt{-1}\Theta_h.$$We will be mainly concerned with either the trivial line bundle $L$ endowed with the Hermitian metric $e^{-qu}$ (ie. $h \equiv 1$) for some $q>0$ and we will use the notation $L_{qu}$ for this data, or the canonical bundle $K_X$ endowed with the Hermitian metric $e^{-qu}\omega^{-2}$ (ie. $h = \omega^{-2}$). It may sometimes be convenient to consider the latter as the bundle $L_{qu}\otimes K_X$ where the $K_X$ factor has the fixed metric $\omega^{-2}$. In either case, for the above choices of reference hermitian metrics, we will use the abridged notation $\zeta_\ep(s,p,u)$ and $\tilde\zeta_\ep(s,p,u)$ for the $(1,1)$ currents, and $\n s\n_u$  for the norms.

\begin{lem}\label{lem:currents}
Suppose $h_u=e^{-u}h$ where $u$ is Lipschitz and the curvature satisfies $\sqrt{-1}\Theta_{h_u}\geq \al$  as an order-zero current, where $\al$ is some smooth closed real $(1,1)$ form (not necessarily positive). Let $\ga>0$ be a real number. Then in the sense of currents, we have the following identities:
\begin{align*}
 \partial \Vert f \Vert_{h_u}^{2\ga} &= -\ga\Vert f \Vert_{h_u}^{2\ga} \partial u +\ga\Vert f \Vert_{h_u}^{2\ga-2} \langle f, \partial f\rangle_{h_u}, \\
  \partial \log(\epsilon^2+ \Vert f \Vert_{h_u}^{2\ga}) &= \frac{-\ga\Vert f \Vert_{h_u}^{2\ga} \partial u +\ga\Vert f \Vert_{h_u}^{2\ga-2} \langle f, \partial f\rangle_{h_u}}{\epsilon^2+\Vert f \Vert_{h_u}^{2\ga}},\\
 \ddbar \log (\Vert f \Vert_{h_u}^{2\ga} +\ep^2) + \ga \sqrt{-1}\Theta_{h_u} &= \ga\frac{ \epsilon^2 \sqrt{-1}\Theta_{h_u}}{ \Vert f \Vert^{2\ga}_{h_u} + \ep^2} +\ga^2\frac{\epsilon^2\Vert f \Vert_{h_u}^{2\ga-2} \nabla^{1,0} f \wedge \nabla^{0,1} f^*}{(\epsilon^2+\Vert f \Vert_{h_u}^{2\ga})^2},
\end{align*}
 and
  $$\ddbar \Vert f \Vert_{h_u}^{2\ga} + \ga\Vert f \Vert_{h_u}^{2\ga} \sqrt{-1}\Theta_{h_u} =  \ga^2 \Vert f \Vert_{h_u}^{2\ga-2}  \nabla^{1,0} f \wedge \nabla^{0,1} f^*\geq 0,$$ where $f^*$ denotes the conjugate of $f$ with respect to $h_u$. In particular, if $\sqrt{-1}\Theta_h\geq 0$, then $$\tilde \zeta(f, \ga,u,h)\geq 0.$$
    \end{lem}
  \begin{proof}
  The proof is standard, but given the important role that these identities play in our proof of the B\'ezout estimates, we include a short argument. We begin with the observation that the right hand sides in the above identities are all locally integrable. The key point is  that $|f|^{2\ga-2}\partial f\wedge \overline{\partial f} \wedge\omega_{Euc}^{n-1}$ is locally in $L^1$ if $\ga>0$. This is of course trivial if $\ga\geq 1$. If $\ga\in (0,1)$, one can prove this by replacing $f$ with a Weierstrass polynomial, and using Fubini's theorem. In fact one can prove that $|f|^{2\ga - 2}|\partial f|^2$ is in $L^{p}_{loc}$ for some $p = p(\ga)>1$. By a similar argument one can show that $|f|^{2\ga - 1}|\partial f|$ is also in $L^p_{loc}$ for some $p>1$.

  We will now prove the third identity; the rest of the identities can be similarly proved.  Without loss of generality we may assume that $\ep = 1$. Let $\chi$ be a $(n-1,n-1)$ form with compact support. In fact without loss of generality, for instance by using a partition of unity, we may also assume that $\chi$ is supported in a single coordinate neighbourhood $U$ over which $L$ can be trivialized. In particular it makes no difference to the proof if we simply assume $L$ is trivial to begin with and that $h\equiv 1$. Moreover, since $\al = \ddbar \theta$ on $U$ for some smooth function $\theta$, we may as well assume, by modifying the metric $h_u$, that $\sqrt{-1}\Theta_{h_u}\geq 0$.  We introduce the following notation: Let $v := \log (1+ \n f\n^{2\ga})$ where the norm is with respect to the metric $e^{-u}$ (recalling that $h\equiv 1$), and let $$\zeta := \ddbar v + \ga\sqrt{-1}\Theta_{h_u} = \ddbar v + \ga\ddbar u.$$Note that by Bedford-Taylor theory, $$\int_X \zeta \wedge\chi= \int_X (v + \ga u) \ddbar\chi.$$Now, there are two sources of non-smoothness: from $|f|^{2\ga}$, and of course from the metric potential $u$ itself. Let $u_\delta$ be a smoothing of $u$ in $U$ such that $u_\delta\rightarrow u$ uniformly on compact subsets of $U$,  and also in $W^{1,p}$ for all $p>1$. We now set $$h_\delta = e^{-u_\delta},~ f_{\delta,\eta}  = \sqrt{\n f\n_{h_\delta}^2 + \eta^2} \text { and } v_{\delta,\eta} = \log(1+ f_{\delta,\eta}^{2\ga}).$$ and let $\zeta_{\delta}$ denote the corresponding $(1,1)$ form given by $$\zeta_{\delta,\eta} = \ddbar v_{\delta,\eta} + \ga\ddbar u_\delta.$$ Then $$\int_X\zeta\wedge\chi = \lim_{\delta\rightarrow 0}\lim_{\eta\rightarrow 0} \int_X (v_{\delta,\eta} + \ga u_\delta)\wedge \ddbar\chi = \lim_{\delta\rightarrow0}\lim_{\eta\rightarrow 0}\int_X\zeta_{\delta,\eta}\wedge\chi .$$ By an elementary calculation we can write $$\int_X\zeta_{\delta,\eta}\wedge \chi = P_{\delta,\eta} + Q_{\delta,\eta} + R^{(1)}_{\delta,\eta} + R^{(2)}_{\delta.\eta},$$ where
  \begin{align*}
  P_{\delta,\eta}&= \ga\int_X \frac{1}{1+f_{\delta,\eta}^{2\ga}}\ddbar u_{\delta} \wedge\chi,\\
  Q_{\delta,\eta} &= \ga^2\int_X \frac{f_{\delta,\eta}^{2\ga-2}}{(1+f_{\delta,\eta}^{2\ga})^2}\nabla^{1,0}_\delta f\wedge \nabla^{0,1}f^{*_{\delta}} \wedge\chi \text{ and }\\
  R^{(1)}_{\delta,\eta} &= \ga\eta^2\int_X\ga\frac{f_{\delta,\eta}^{2\ga-2}}{1+f_{\delta,\eta}^{2\ga}}\ddbar u_\delta\wedge\chi \\
  R^{(2)}_{\delta,\eta} &= \ga\eta^2 \int_X\frac{f_{\delta,\eta}^{2\ga-4}(f_{\delta,\eta}^{2\ga} + \ga-1)}{(1+f_{\delta,\eta}^{2\ga})^2}\nabla^{1,0}_\delta f\wedge \nabla^{0,1}f^{*_\delta}\wedge\chi,
  \end{align*}
  By Bedford-Taylor theory, since $u_\delta$ decreases, and converges uniformly, to $u$ we have that $$\lim_{\delta\rightarrow 0}\lim_{\eta\rightarrow0}P_{\delta,\eta} = \ga\int_X \frac{1}{1+\n f\n^{2\ga}}\ddbar u \wedge \chi = P_{0,0}.$$ Next we proved that $\lim_{\delta\rightarrow 0}\lim_{\eta\rightarrow 0}Q_{\delta,\eta} = Q_{0,0}$. Without loss of generality we may assume that $\ga\in (0,1)$, since if $\ga\geq 1$ then the integrands converge uniformly and the claim is trivial.  First we note that for a fixed $\delta$,  since $\ga \in (0,1)$ we have $$ \Big|\frac{f_{\delta,\eta}^{2\ga-2}}{(1+f_{\delta,\eta}^{2\ga})^2}\nabla^{1,0}_\delta f\wedge \nabla^{0,1}f^{*_{\delta}} \wedge\chi \Big| \leq  \Big|\frac{f_{\delta,0}^{2\ga-2}}{(1+f_{\delta,0}^{2\ga})^2}\nabla^{1,0}_\delta f\wedge \nabla^{0,1}f^{*_{\delta}} \wedge\chi \Big|.$$ The right hand side is precisely the integrand of $Q_{\delta,0}$ and is  integrable by the remarks made at the beginning of the proof, and hence by dominated convergence theorem we have that $$\lim_{\eta\rightarrow 0}Q_{\delta,\eta} = Q_{\delta,0}.$$  Next, $$ |Q_{\delta,0} - Q_{0,0}| \leq \ga^2(I + J),$$
  where
  \begin{align*}
  I &= \int_X \Big|\frac{f_{\delta,0}^{2\ga-2}}{(1+f_{\delta,0}^{2\ga})^2}-\frac{\n f\n^{2\ga-2}}{(1+\n f\n^{2\ga})^2}\Big||\nabla^{1,0}_\delta f\wedge \nabla^{0,1}f^{*_{\delta}} \wedge\chi |\text{ and }\\
  J &= \int_X \Big|\frac{\n f\n^{2\ga-2}}{(1+\n f\n^{2\ga})^2}\Big| |\nabla^{1,0}_\delta f\wedge \nabla^{0,1}f^{*_{\delta}} \wedge\chi - \nabla^{1,0}f\wedge \nabla^{0,1}f^* \wedge\chi |
  \end{align*}
  To estimate the first integral, we recall that  $u_\delta$ converges uniformly to $u$ and decreases to $u$. It is then easy to see that for any $\ep >0$ there exists a $\delta(\ep)>0$ such that for all $\delta < \delta(\ep)$ we have  $$ \Big|\frac{f_{\delta,0}^{2\ga-2}}{(1+f_{\delta,0}^{2\ga})^2}-\frac{\n f\n^{2\ga-2}}{(1+\n f\n^{2\ga})^2}\Big| \leq \ep \n f\n^{2\ga-2}. $$ Indeed, if we let $b = f_{\delta,0}^2 = \n f\n^2_{\delta}$ and $a = \n f\n^2$. Then
  \begin{align*}
  \Big|\frac{f_{\delta,0}^{2\ga-2}}{(1+f_{\delta,0}^{2\ga})^2}-\frac{\n f\n^{2\ga-2}}{(1+\n f\n^{2\ga})^2}\Big| &= \Big|\frac{b^{\ga-1}}{(1+b^\ga)^2} - \frac{a^{\ga-1}}{(1+a^\ga)^2}\Big|\\
  &\leq |b^{\ga-1}(1+a^\ga)^2 - a^{\ga-1}(1+b^\ga)^2|\\
  &= a^{\ga - 1}\Big|\frac{a^{1-\ga}}{b^{1-\ga}}(1+a^\ga)^2 - (1+b^\ga)^2\Big|
  \end{align*}
  Now the RHS can be made smaller than $\ep$ uniformly on the support of $\chi$ in $U$ for all $\delta < \delta(\ep)$ since $u_\delta$ converges to $u$ uniformly.

  On the other hand,
  \begin{align*}
  |\nabla^{1,0}_\delta f \wedge \nabla^{0,1}f^{*_\delta} \wedge\chi| &= e^{-u_\delta}(\partial f - f\partial u_\delta)\wedge\overline{(\partial f - f\partial u_\delta)}\wedge\chi\\
  &\leq C(|\partial f|^2 +|f| |\partial f||\partial u_\delta| + |f|^2|\partial u_\delta|^2)\omega\wedge\chi.
  \end{align*}
  for some constant $C>0$ independent of $\delta$.  So $$I \leq C\ep(I_1+I_2+I_3),$$ where $$I_j = \int_X|f|^{2\ga-2}|f|^j|\partial u_\delta| |\partial f|^{2-j}\omega\wedge\chi.$$ We estimate $I_1$; the other two terms can be similarly estimated. By the discussion at the beginning of the proof there exists a $p>1$ such that  $|f|^{2\ga-1}|\partial f|$ is in $L^p_{loc}$. By Holder if $p^*$ is the conjugate exponent of $p$, then $$I_1\leq \Big(\int_{\mathrm{Supp}(\chi)}|f|^{(2\ga-1)p}|\partial f|^{p}\Big)^{1/p}\Big(\int_{\mathrm{Supp}(\chi)}|\partial u_\delta|\Big)^{1/p^*}.$$ The second integral on the right above is also bounded uniformly independent of $\delta$ since $u_{\delta}\rightarrow u$ in $W^{1,p^*}(U)$.  $J$ can also be similarly estimated. Finally, we prove that for each $\delta$, $$\lim_{\eta\rightarrow 0}R^{(j)}_{\delta,\eta} = 0, ~ j=1,2.$$ Once again, it is enough to prove this for $\ga \in (0,1)$ since if $\ga\geq 1$ then the proofs are even simpler. For a fixed $\delta$,  $$ R^{(1)}_{\delta,\eta} \leq C_{\delta}\eta^2\int_X \frac{1}{(\n f\n_{h_\delta}^2+\eta^2)^{1-\ga}}\omega\wedge\chi \leq C_\delta\eta^{2\ga}\xrightarrow{\eta\rightarrow 0}0.$$ For the second remainder term, for fixed $\delta>0$, we note (by an argument similar to above estimates used in bounding $I_1$) that $$|\nabla_\delta^{1,0}f\wedge\nabla^{0,1}f^{*_\delta}\wedge\chi| \leq C_\delta(|f|^2 + |\partial f|^2)\omega\wedge\chi.$$ Moreover by choosing $C_\delta$ large enough we also have that $$\n f\n_{h_\delta}^2 + \eta^2 \leq C_\delta(|f|^2 + \frac{\eta^2}{C_\delta}).$$ Then we have
  \begin{align*}
  R^{(2)}_{\delta,\eta} &\leq C_\delta\ga\eta^2\int_X (|f|^2 + \frac{\eta^2}{C_\delta})^{\ga-2}(|f|^2 + |\partial f|^2)\omega\wedge\chi \\
  &\leq C_\delta\ga\eta^2\Big(\int_X (|f|^2 + \frac{\eta^2}{C_\delta})^{\ga-1}\omega\wedge \chi + \int_X  (|f|^2 + \frac{\eta^2}{C_\delta})^{\ga-2}|\partial f|^2\omega\wedge\chi.
  \end{align*}
  The first integral above is obviously bounded above by $\eta^{2\ga - 2}$ since $\ga < 1$, and hence the first term is bounded by $\eta^{2\ga}$. On the other hand for the second integral, setting $\tilde\eta = \eta/C_\delta$, replacing  $f$ by its Weierstrass polynomial and using Fubini, it is not difficult to see that $$ \int_X  (|f|^2 + \tilde\eta^2)^{\ga-2}|\partial f|^2\omega\wedge\chi \leq C_\delta \eta^{2\ga - 2}.$$ And so for a possibly larger constant $C_\delta$, $$R_{\delta,\eta}^{(2)}\leq C_\delta\eta^{2\ga}\xrightarrow{\eta\rightarrow0}0.$$

  \end{proof}

  An important consequence that we shall use repeatedly is the following.

\begin{cor}\label{cor:zeta-positive}
Let $(L,h)$ be a Hermitian line bundle on $X$ and $u\in \mathrm{Lip}(X)\cap \mathrm{PSH}(X)$. Then for any $\ep,p,q>0$ and any $s\in H^0(X,L)$, $$\zeta_\ep(s, p,qu,h)\geq -\frac{\n s\n_{qu}^p}{\n s\n^p_{qu} + \ep^2}\sqrt{-1}\Theta_h.$$ In particular, if  $\sqrt{-1}\Theta_h\leq 0$ then $$\zeta_\ep(s, p,qu,h)\geq 0.$$
\end{cor}
\begin{proof}
From the Lemma above we see that
\begin{align*}
\zeta_\ep(s,p,qu):&= \ddbar\log(\n s\n_{qu}^p + \ep^2) + \frac{p}{2}\sqrt{-1}\Theta_{h_{qu}} - \frac{p}{2}\sqrt{-1}\Theta_h\\
&\geq \frac{\ep^2p}{2}\frac{\sqrt{-1}\Theta_{h_{qu}}}{\n s\n^p_{qu}+\ep^2}  - \frac{p}{2}\sqrt{-1}\Theta_h\\
&\geq -\frac{\n s\n_{qu}^p}{\n s\n^p_{qu} + \ep^2} \sqrt{-1} \Theta_h,
\end{align*}
where we used the fact that $\sqrt{-1} \Theta_{h_{qu}} \geq \sqrt{-1} \Theta_h$ since $u$ is psh.
\end{proof}
We will also need the following specific case.
\begin{cor}\label{lem:ric-zeta-upperbound}
Let $\sigma\in H^0(X,K_X)$ and $u\in \mathrm{Lip}(X)\cap \mathrm{PSH}(X)$. Then $$\zeta_\ep(\sigma,2,u) \geq \frac{\n \sigma\n^2_{u}}{\n \sigma\n_{u}^2 + \ep^2}Ric.$$
In particular, if $Ric\geq 0$, then $\zeta_\ep(\sigma,2,u)\geq 0$, and hence by Fatou's lemma, $$\int_XRic^2 \leq
\liminf_{\vep\rightarrow 0^+} \int_X\zeta_\ep(\sigma,2,qu)\wedge\zeta_\ep(\sigma,2,qu),$$ where the wedge product on the right is interpreted in the Bedford-Taylor sense.

\end{cor}
 We also need a version of the final equality above for vector bundles of higher rank.
 \begin{lem}\label{lem:bochner-Weitzenbock-distributional}Let $(E,H)$ be a Hermitian holomorphic vector bundle and $\xi$ a holomorphic section of $E$. Suppose that for some $A>0$, the curvature of $H$ is less than $A\omega$ in the Griffiths sense, then $$\ddbar \Vert \xi \Vert \geq -A \Vert \xi \Vert\omega$$ in the distributional sense.
 \end{lem}

 \begin{proof}We will use the following standard Bochner-Weitzenb\"ock formula: $$\ddbar \n \xi \n^2 = \nabla^{1,0}\xi \wedge \nabla^{0,1}\xi^* - \langle \Theta\xi,\xi\rangle. $$
We need to show that for any positive $(n-1,n-1)$ form $\eta$, we have
\begin{equation}\label{lem:bochner-Weitzenbock-distributional-cut-off}
\int_X\n \xi \n\ddbar\eta \geq -A\int_X\n\xi\n\omega \wedge \eta.
\end{equation}

Let $f_\ep = \sqrt{\n\xi\n^2 + \ep^2}$. Then
\begin{align*}
\ddbar f_\ep &= \frac{\ddbar\n\xi\n^2}{2f_\ep} - \frac{\sqrt{-1}\partial\n\xi\n^2 \wedge\overline{\partial}\n\xi\n^2}{4f^3_\ep}\\
&= -\frac{\langle\Theta\xi,\xi\rangle}{2f_\ep} + \frac{2f^2_\vep\nabla^{1,0}\xi \wedge \nabla^{0,1}\xi^* -\sqrt{-1}\partial\n\xi\n^2 \wedge\overline{\partial}\n\xi\n^2 }{4f_\ep^3}\\
&\geq  -\frac{\langle\Theta\xi,\xi\rangle}{2f_\ep} + \frac{2\n\xi\n^2\nabla^{1,0}\xi \wedge \nabla^{0,1}\xi^* -\sqrt{-1}\partial\n\xi\n^2 \wedge\overline{\partial}\n\xi\n^2 }{4f_\ep^3}
\end{align*}
Computing at a point $p$ in a normal trivialization for the metric $H$ in which $H(p)$ is diagonal and $dH(p) = 0$, and at the same time computing in normal coordinates at $p$, we see that $\n\xi\n\nabla^{1,0}\xi \wedge \nabla^{0,1}\xi^*$ is a positive $(1,1)$ form and moreover by Cauchy-Schwarz inequality, $$\n\xi\n^2\nabla^{1,0}\xi \wedge \nabla^{0,1}\xi^* -\sqrt{-1}\partial\n\xi\n^2 \wedge\overline{\partial}\n\xi\n^2\geq 0,$$ and hence by our hypothesis on the curvature, $$\ddbar f_\ep \geq -A\frac{\n\xi\n^2}{\sqrt{\n\xi\n^2 + \ep^2}}\omega.$$ Multiplying by $\eta$ and integrating by parts $$\int_X f_\ep\ddbar\eta\geq -A\int_X\frac{\n\xi\n^2}{\sqrt{\n\xi\n^2 + \ep^2}}\omega\wedge \eta.$$ Letting $\ep\rightarrow0$ and using dominated convergence we prove \ref{lem:bochner-Weitzenbock-distributional-cut-off}

\end{proof}

Finally, we will need to control various intersection numbers using integrals involving $\zeta_\ep$ and $\tilde\zeta_\ep$. The discussion that follows extends to higher dimensions, but for simplicity we confine ourselves to the case of complex surfaces. Let $X$ be a (possibly non-compact) complex surface. The proof rests on the following two propositions, the first of which is a version of the Poincar\'e--Lelong formula.

\begin{prop}\label{prop:poincare-lelong} Let $(L,h)$ be a Hermitian line bundle and $u\in \Lip(X)\cap PSH(X)$. Let $s\in H^0(X,L)$ be such that $S = \{s = 0\}$ is an irreducible hypersurface with multiplicity $m$. Then for any $(n-1,n-1)$ continuous form $\eta$ with compact support, we have $$m\int_S \eta = \frac{1}{\pi p}\lim_{\ep\rightarrow 0} \int_X \tilde{\zeta}_{\ep}(s,p,u,h)\wedge\eta.$$
\end{prop}
When $u$ and $\eta$ are smooth, this result is of course standard. When $\eta$ is smooth, one can prove the above by smoothing. In the general case, we approximate $\eta$ uniformly by smooth forms with compact support $\eta_l$. If $T_k$ is any sequence of order zero currents converging weakly to $T$, then $\vert T_k(\eta)-T(\eta)\vert = \vert T_k(\eta_l)+T_k(\eta-\eta_l)-T(\eta_l)+T(\eta_l-\eta) \vert$. Since $\Vert T_k \Vert \leq C$, we can easily complete the proof. \\

Before we state the main result on estimating intersection numbers, we recall some basics of intersection theory of divisors on surfaces. Let $X$ be a smooth surface and $S = \{s = 0\}$ and $T = \{t = 0\}$ be two divisors. In general $s,t$ are sections of a line bundle, but for local intersection, we may assume that locally near a point $x\in X$, these are holomorphic functions. In particular $f,g$ represent elements in the local ring $\sO_{X,x}$. If $S$ and $T$ have no component in common, then the intersection multiplicity of $S$ and $T$ is defined by $$i_x(S,T) = \dim_\CC \sO_{X,x}\slash (f,g).$$ It is standard fact that this is well defined and finite. More concretely in our situation if locally near $x$,
\begin{equation}\label{eq:T-decomposition}
T = \sum_j a_j C_j
\end{equation} is the decomposition of $T$ into its prime components, and $\nu_j:\tilde C_j\rightarrow C_j$ are normalisations, then  $$i_x(S,T) = \sum_{j: x\in C_j}a_j \sum_{y\in \nu_j^{-1}(x)}\mathrm{ord}_{y}(\nu_j^*s),$$ where $y_j\in \tilde C_j$ is a point lying above $x$ and $\mathrm{ord}$ is the order of vanishing.

Note that since $S$ and $T$ have no common components, the set $\{x\in X~|~ i_x(S,T) \neq 0\}$ is a discrete set. So for any compact set $K\subset X$ the local intersection product and the set theoretic count given respectively by $$S\cdot_K T:= \sum_{x\in K}i_x(S,T)\text{ and } N_K(S,T) = \#(K\cap \mathrm{Supp}(S)\cap \mathrm{Supp}(T)).$$ Note that if $T$ is given as above, then $$S\cdot_K T = \sum_ja_j\deg(\mathrm{div}(\nu_j^* s)).$$We can similarly define the total intersection numbers and set theoretic counts $$S.T,~ N(S,T) \in [0,\infty].$$ We denote the currents of integration by $[S]$ and $[T]$ and define their wedge product by $$[S]\wedge [T] = \sum_{x\in X}i_x(S,T)\delta_x$$ as a measure. Here $\delta_x$ is the Dirac mass at $x$. Then for any compactly supported smooth function $\chi\in C^\infty_c(X)$, suppose $T$ is given by the expression \eqref{eq:T-decomposition}, then we have that
\begin{align*}
\int_X\chi [S]\wedge[T] :&= \sum_{x\in X}\chi(x)\sum_{j:x\in C_j} a_j\sum_{y\in \nu_j^{-1}(x)} \mathrm{ord}_{y}(\nu_j^*s)\\
&=\sum_{j}a_j\sum_{x\in C_j}\sum_{y\in \nu_j^{-1}(x)}\mathrm{ord}_{y}(\nu_j^*s)\chi(\nu_j(y)).
\end{align*}

\begin{prop}\label{prop:intersectionnumberestimates} Let $X$ be a smooth complex surface, and let $S=(s=0)$ and $T=(t=0)$ be effective divisors with no common irreducible component. Let $h$ and $k$ be smooth Hermitian metrics on $\sO(S)$ and $\sO(T)$ satisfying
\[
 -\sqrt{-1}\Theta_h \ge 0,
 \qquad
 -\sqrt{-1}\Theta_k \ge 0,
\]
and let $u,v\in \PSH(X)\cap \Lip(X)$. Set $h_u=e^{-u}h$ and $k_v=e^{-v}k$. Then the following estimate holds. $$N(S,T) \leq S.T \leq \frac{1}{\pi^2pm}\liminf_{\delta\rightarrow 0}\liminf_{\ep\rightarrow 0}\int_{X}\zeta_\ep(s,p,u,h)\wedge \zeta_\delta(t,m,v,k).$$

\end{prop}

We will need the following elementary observation.

\begin{lem}\label{lem:wedge-with-divisor}
Let $f$ be a locally bounded PSH function and let $\zeta = \ddbar f$. Let $T$ be an effective divisor. Then $$\zeta \wedge [T] := \ddbar(f[T])$$ is well defined and a positive Radon measure. Moreover, suppose $\chi\in C^\infty_c(X)$ is a non-negative function and that $T$ is given by the decomposition \eqref{eq:T-decomposition} in the support of $\chi$, then $$\int_X\chi \zeta\wedge [T] = \sum_j a_j\int_{\tilde C_j}\nu_j^*\chi \ddbar \nu_j^*f,$$ where $\nu_j:\tilde C_j\rightarrow C_j$ as above are the normalisations.
\end{lem}
\begin{proof}
The proof of the first part is a standard Bedford-Taylor theory construction \cite{BT} (cf. \cite[Proposition 1.7]{Ko}). For the second part, we note that for any $(1,1)$ form $\eta$ with compact support contained in the support of $\chi$, we have:
\begin{align*}
\langle f[T],\eta\rangle &= \sum_j a_j \langle f[C_j],\eta\rangle\\
&= \sum_j a_j\int_{C_{j,reg}}f\eta\\
&=\sum_j a_j \int_{\tilde C_{j}}\nu_j^* f\nu_j^*\eta.
\end{align*}
Now applying this to $\eta = \ddbar\chi$ we have that
\begin{align*}
\int_X\chi \zeta\wedge [T] &= \langle f[T],\ddbar\chi\rangle\\
&= \sum_j\int_{\tilde C_j}\nu_j^*f \nu_j^*\ddbar\chi\\
&= \sum_j\int_{\tilde C_j} \nu_j^*\chi\ddbar \nu_j^* f,
\end{align*}
where note that $\nu_j^*f$ is a locally bounded subharmonic function and hence the final integral is well defined.
\end{proof}

\begin{proof}[Proof of Proposition \ref{prop:intersectionnumberestimates}]
It is enough to show that for any smooth function $0\leq \chi\leq 1$ with compact support $S\cdot_K T$ is less than the integral on the right where we set $K = \chi^{-1}(1).$ By the above calculation, if $T$ is given on the support of $\chi$ by the decomposition in \eqref{eq:T-decomposition}, then $$S\cdot_K T \leq \int_X \chi [S]\wedge[T] =\sum_{j}a_j\sum_{x\in C_j}\sum_{y\in \nu_j^{-1}(x)}\mathrm{ord}_{y}(\nu_j^*s)\chi(\nu_j(y)).$$ From the one dimensional Poincar\'e-Lelong formula, the final two summations can be replaced by a limit and we have
\begin{align*}
S\cdot_K T&\leq \frac{1}{\pi p}\lim_{\ep \rightarrow 0^+}\sum_j a_j\int_{\tilde C_j}\nu_j^*\chi \nu_j^*\Big(\zeta_\ep(s,p,u,h) + \sqrt{-1}\Theta_h\Big) \\
&\leq \frac{1}{\pi p}\lim_{\ep \rightarrow 0^+}\sum_j a_j\int_{\tilde C_j}\nu_j^*\chi \nu_j^*\zeta_\ep(s,p,u,h),
\end{align*}
where we use the hypothesis that $\sqrt{-1}\Theta_h\leq 0.$  Now let $$f_\ep = \log(\n s\n_{u}^p + \ep^2) + \frac{p}{2}u,$$ so that $\zeta_\ep(s,p,u,h) = \ddbar f_\ep$. Note that since $\sqrt{-1}\Theta_h\leq 0$, by Corollary \ref{cor:zeta-positive}, $f_\ep$ is PSH. Moreover it is also locally bounded. Then by Bedford-Taylor theory, the measures $$\zeta_\ep(s,p,u,h)\wedge\zeta_\delta(t,m,v,k) = \ddbar(f_\ep \zeta_\delta(t,m,v,k)) \rightharpoonup \pi m\ddbar(f_\ep([T] - \frac{\sqrt{-1}}{2\pi}\Theta_k))$$ weakly converge as $\delta\rightarrow 0^+$. So pairing with $\chi$ we have
\begin{align*}\label{eq:intersection-eq}
\lim_{\delta\rightarrow 0^+}\int_X\chi \zeta_\ep(s,p,u,h)\wedge\zeta_\delta(t,m,v,k) &= \pi m\int_X \chi \ddbar(f_\ep [T])  - \frac{m}{2}\int_X\chi \zeta_\ep \wedge \sqrt{-1}\Theta_k.
\end{align*}
But now since $\sqrt{-1}\Theta_k\leq 0$, we obtain the following estimate:
\begin{equation}\label{eq:int-number}\lim_{\delta\rightarrow 0^+}\int_X\chi \zeta_\ep(s,p,u,h)\wedge\zeta_\delta(t,m,v,k) \geq  \pi m\int_X \chi \ddbar(f_\ep [T])
\end{equation}
We now compute the integral on the right. Once again, suppose $T$ is given by the decomposition \eqref{eq:T-decomposition} on $\mathrm{Supp}(\chi)$, then by Lemma \ref{lem:wedge-with-divisor} we have $$\int_X \chi \ddbar(f_\ep [T])  = \sum_j a_j\int_{\tilde C_j} \nu_j^*\chi \ddbar \nu_j^* f_\ep = \sum_j a_j\int_{\tilde C_j}\nu_j^*\chi\nu_j^*\zeta_\ep(s,p,u,h).$$ Combining with our earlier estimate for $S\cdot_KT$ we then have that
\begin{align*}
S\cdot_K T&\leq \frac{1}{\pi p} \lim_{\ep\rightarrow 0^+} \int_X \chi \ddbar(f_\ep [T]) \\
&\leq \frac{1}{\pi^2 pm}\liminf_{\ep\rightarrow 0^+}\lim_{\delta\rightarrow 0^+}\int_X \chi \zeta_\ep(s,p,u,h)\wedge\zeta_\delta(t,m,v,k)\\
&\leq \frac{1}{\pi^2 pm}\liminf_{\ep \rightarrow 0^+}\liminf_{\delta\rightarrow 0^+}\int_X \zeta_\ep(s,p,u,h)\wedge\zeta_\delta(t,m,v,k).
\end{align*}


\end{proof}

Before we prove this we need the following Lemma.

\subsection{Heat flow estimates} Let $(X^m,g)$ be a complete Riemannian manifold satisfying $\Rc_g\geq 0$. Then it is well known that there is a positive, symmetric heat kernel $H(x,y,t)$. Moreover the heat kernel is unique and stochastically complete ie. $$\int_X H(x,y,t)\,dy = 1.$$  It is also well known that if $\phi:X\rightarrow \RR$ satisfies a sub-exponential growth at infinity, then the heat equation $$\begin{cases} \frac{\partial u}{\partial t} = \Delta u\\ u(x,0) = \phi\end{cases}$$ has a unique solution given by the representation formula $$u(x,t) = \int_X H(x,y,t) \phi(y)\,dy.$$

We will need the following theorem proved by Ni-Tam \cite{NT}.
\begin{thm} [Ni-Tam]
Let $(X^n,\omega)$ be a complete non-compact K\"ahler manifold with $BK >  0$. Let $\phi$ be a continuous  plurisubharmonic function on $X$ satisfying

$$\vert \phi  (x)\vert \le C e^{ar^2(x)}$$
for some positive constants $a, C$ and $r(x) $ denotes the distance to a fixed point. Then there exists $T_0>0$ depending only on $a$ such that the solution  $u(x,t)$ to the heat equation with initial condition $u(x,0) = \phi(x)$ is a smooth plurisubharmonic function defined on $X \times (0, \ T_0)$. 

\end{thm}

We will also need the following important observation from \cite{CZ3}.

\begin{lem} [Chen-Zhu]
Let $(X^n,\omega)$ be a complete non-compact K\"ahler manifold with $BK >  0$. Let $\phi$ be a Lipschitz  plurisubharmonic function on $X$ with Lipschitz constant $1$. Then the solution  $u(x,t)$ to the heat equation with initial condition $u(x,0) = \phi(x)$ exists on $X \times (0, \infty)$ and the smooth plurisubharmonic function $u(x,1)$ has bounded Levi form. i.e there exists $C>0$ such that
$$ \vert \partial \bar \partial u(x,1) \vert \le C.$$

\end{lem}

Finally, we need the following crucial estimate on comparing $u(x,t)$ to $\phi$ assuming $\phi$ is uniformly Lipschitz. Note that for our applications it is crucial that the constant multiplying $\phi$ is one. The following lemma was proven in \cite{DPS} but we include the proof here for the convenience of the reader.

\begin{lem}\label{lem:heat-est}
Let $(X^m,g)$ be a complete Riemannian manifold satisfying $\Rc_g\geq 0$. Let $u(x,t)$ be the heat flow with the initial function $\phi(x)$. Suppose $\phi$ is Lipschitz with Lipschitz constant $A$. Then there exists a dimensional constant $C(m)$ such that $$u(x,t) \leq \phi(x) + AC\sqrt{t}.$$
\end{lem}
\begin{proof}
 From the representation formula, the fact that $\phi$ is Lipschitz and stochastic completeness, it follows that $$u(x,t) \leq \phi(x) + A\int_X H(x,y,t)d(x,y)\,dy.$$ So it suffices to estimate the integral. By the fundamental Li-Yau gradient estimates (cf. \cite[Corollary 3.1]{LY}), we have the following Gaussian estimate: $$H(x,y,t) \leq \frac{C_1(n)}{|B(x,\sqrt{t})|}e^{-c\frac{d^2(x,y)}{t}},$$ for some $c<1/4$. Now for integers $k\geq 0$, consider the annuli $$A_k = B(x,(k+1)\sqrt{t})\setminus B(x,k\sqrt{t}).$$ Then
\begin{align*}
\int_X H(x,y,t)d(x,y)\,dy &\leq \sum_k\int_{A_k}H(x,y,t)d(x,y)\,dy\\
&\leq \frac{C_1(n)\sqrt{t}}{|B(x,\sqrt{t})|} \sum_k (k+1) e^{-ck^2}|A_k|.
\end{align*}
But now by the Bishop-Gromov inequality, $$\frac{|A_k|}{|B(x,\sqrt{t})|} \leq \frac{|B(x,(k+1)\sqrt{t})|}{|B(x,\sqrt{t})|} \leq \omega_m(k+1)^m,$$ and so $$\int_X H(x,y,t)d(x,y)\,dy \leq C_1(m)\omega_m\sqrt{t}\sum_{k=1}^\infty (k+1)^{m+1}e^{-ck^2} \leq C(m)\sqrt{t}.$$

\end{proof}

\subsection{Cut-off functions} An important application of the heat flow technique is to construct good cut-off functions. We have the following standard result in that direction.
\begin{lem}\label{lem:cut-off}
Let $(X,\omega)$ satisfy $\sec_\omega \geq 0$. Fix $o\in X$. Then there exist $0< \theta < 1$, $A>0$ and $R_0>0$ such that the following holds: for all $R>R_0$ there exist a smooth function $\chi_R:X\rightarrow [0,1]$ having the following properties:
\begin{enumerate}
\item $\chi_R \equiv 1$ on $B(o,\theta R)$ and $\mathrm{Supp}(\chi_R)\subset B(o,\theta^{-1}R).$
\item There exists a constant $A$ such that $$|\nabla \chi_R|, |\ddbar \chi_R| \leq \frac{A}{R}. $$
\end{enumerate}
\end{lem}
\begin{proof}
This is standard, so we only sketch the proof. We proceed as in \cite{CZ4}. Let $b(x)$ be Busemann function given by $$b(x) = \sup_{\ga}\lim_{t\rightarrow \infty}(t - d(x,\ga(t)),$$ where the supremum is taken over all geodesic rays $\ga$ emanating from a fixed point $o\in X$. Then it is well known that $b$ is uniformly Lipschitz and plurisubharmonic. Moreover, there exists a constant $C_1>0$ such that $$\frac{1}{2}d(x,o) - C_1 \leq b(x) \leq d(x,o).$$  Let $u(x,t)$ solve the heat equation with $u(x,0) = b(x)$, and let $\eta(x) = u(x,1)$. Then there exists a constant $C>1$ such that
\begin{itemize}
\item $C^{-1}(1+ d(x,o))\leq \eta(x) \leq C(1+d(x,o))$.
\item $|\nabla \eta|, |\ddbar\eta| < C.$
\end{itemize}
Now let $\chi:\RR\rightarrow [0,1]$ be the usual cut-off function such that $\chi \equiv 1$ on $t\leq 1$ and $\mathrm{Supp}(\chi) \subset (-\infty, 2]$. Then $$\chi_R(x) := \chi \Big(\frac{\eta(x)}{R}\Big)$$ does the job with $\theta = (2C)^{-1}$ and $R_0 = 2C$.
\end{proof}

\begin{rem}
It is possible to prove the above Lemma under only the assumption of positive bisectional curvature by working with $u(x,0) = d(x,o)$. In this case, while $u(x,0)$ is not plurisubharmonic one still has the required growth estimates on the time one heat flow, and also the estimates on the gradient and the complex Hessian.
\end{rem}

\begin{rem}
Recall that under the assumption of $Ric \geq 0$, it is well known \cite{CC1} that one can construct such a cut-off function with $R^2|\Delta\chi_R| < C$. Similarly in our situation, by using the Ni-Tam theory, it should be possible to construct a $\chi_R$ with $R^2|\ddbar \chi_R| < C$ under the hypothesis of positive bisectional curvature. For the present work however, the weaker version above suffices.
\end{rem}

\subsection{Mean value inequality and applications} We also need the following basic mean value inequality. We work with a Riemannian manifold $(X^m,g)$ of real dimension $m$.

\begin{prop}
Let $(X^m,g)$ be a complete Riemannian manifold with $\Rc_g \geq0$. Let $f \in C^0(X)$ be non-negative and quasi subharmonic in the distributional sense ie. there exists a constant $A>0$ such that for any compactly supported, non-negative function $\chi\in C^\infty_c(X)$ then $$\int_X (f\Delta\chi + Af\chi)\omega^n \geq 0.$$ Then for any $x\in X$ and any $r>0$ we have $$f(x) \leq 2^m\cdot \frac{e^{\sqrt{A}r} - e^{-\sqrt{A}r}}{\sqrt{A}r}\cdot\frac{1}{|B(x,r)|}\int_{B(x,r)} f.$$
\end{prop}
\begin{proof}

The proof is standard, and we argue as in \cite{CZ2}. Consider the product manifold $\tilde X = X\times \RR$ with the product metric $\tilde g = g + dt^2$. Let $F(x,t) = e^{\sqrt{A}t}f(x)$. Then $F$ is continuous on $\tilde X$ and satisfies $$\tilde \Delta F = \Delta f + Af \geq 0$$ in a distributional sense. Then by the mean value inequality for subharmonic functions (cf. \cite{LS}) we have  $$F(x,0) \leq \frac{1}{|\tilde B ((x,0),r)|}\int_{\tilde B((x,0),r)}F(x,s)\,dx\,ds.$$ To estimate the integral we note that $\tilde B ((x,0),r)$ is contained in the cylinder $B(x,r) \times [-r,r]$ and so $$\int_{\tilde B((x,0),r)}F(x,s)\,dx\,dt \leq \int_{-r}^re^{\sqrt{A}s}\int_{B(x,r)}f(x)\,dx = \frac{e^{\sqrt{A}r} - e^{-\sqrt{A}r}}{\sqrt{A}}\int_{B(x,0)}f(x)\,dx .$$ On the other hand, we also have that $B(x,r/2) \times [-r/2,r/2] \subset \tilde B((x,0),r)$, and so applying volume comparison we have $$|\tilde B((x,0),r)| \geq r |B(x,r/2)| \geq \frac{r}{2^m}|B(x,r)|.$$

\end{proof}

\begin{cor}\label{cor:poly-growth-general}
Let $(X^m,g)$ be a complete, non-compact Riemannian manifold with $\Rc_g \geq0$. Let $f \in C^0(X)$ be a non-negative function such that  $$\Delta f\geq -Af$$ holds in the distributional sense for some constant $A>0$. Fix a point $o\in X$. Then there exists a constant $C>0$ such that $$f(x) \leq C(1+ d(o,x)^m)\int_{B(x,1)}f.$$
\end{cor}
\begin{proof}
Let $R = 1 + d(o,x)$. We apply the above mean value inequality with $r = 1$. Note that $B(x,1)\subset B(x,R)$ and $B(o,1)\subset B(x,R)$ and so by volume comparison, $$|B(x,1)| \geq \frac{|B(x,R)|}{R^m}\geq \frac{|B(o,1)|}{R^m}.$$ Then applying the mean value inequality we prove the Corollary with the explicit constant
 $$C:= 2^m\frac{e^{\sqrt{A}} - e^{-\sqrt{A}}}{\sqrt{A}|B(o,1)|}.$$
\end{proof}

\section{B\'ezout estimates} \label{sec:Bezout}
\subsection{Weights and function spaces} In \cite{DPS}, we proved the following theorem.

\begin{thm}[Theorem 3 in \cite{DPS}]\label{thm:main-ma}
Let $(X^n,\omega)$ be a complete, non-compact $n$-dimensional K\"ahler manifold. Suppose
\begin{enumerate}
\item $BK>0$ and
\item $X$ is strongly Stein in the sense that there exists a smooth exhaustion function $\rho:X\rightarrow \RR$ with uniformly bounded gradient.
\end{enumerate} Then there exists a uniformly Lipschitz, strictly plurisubharmonic function $\phi$ on $X$ such that $$\int_X(\ddbar \phi)^n < \infty.$$ In particular, in view of the theorem of Green and Wu, such a $\phi$ exists on any complete, non-compact K\"ahler manifold with positive sectional curvature.
\end{thm}

 We now let $H(x,y,t)$ be the heat kernel on $(X,\omega)$ with total integral one, and consider
 \begin{align}
 u(x,t) = \int_X H(x,y,t)\phi(y)\,dy.
 \end{align} Then $u(x,t)$  is a  solution to the heat equation with initial condition $u(x,0) = \phi(x)$. We let $\psi(x):= u(x,1)$. Then there exists a constant $A>0$ such that for all $t\in (0,1]$ we have the following estimates:
 \begin{align*}
 |\nabla u_t|^2,~t|\ddbar u_t|^2 \leq A\\
 u_{t}(x) \leq \phi(x) + A\sqrt{t} \text{ and }\\
 \psi(x)\leq u_t(x) + A \sqrt{1-t}.
 \end{align*}
 We will work with both the non-smooth weight $\phi$ and the smooth weight $\psi$. For the proof of our main B\'ezout estimate, we will work with all the weights $u$.

Next, we introduce the relevant function spaces. Let $L$ be the trivial line bundle equipped with the Hermitian metric $e^{-u}$. To emphasise the metric we will sometimes denote the line bundle by $L_u$. For any section $s$ of $L$ we denote by $\n s\n_{u}$ the norm of $s$ with respect to the Hermitian metric $e^{-u}$.  Note that a section of $L_u$ is simply a holomorphic function, but it is helpful to think of it as a section of a line bundle. We define
\begin{align*}\Ga^p_{qu}:= \{s\in H^0(X,L)~\Big |~ \int_X \n s\n_{qu}^p\omega^n:= \int_X |s|^p e^{-\frac{pqu}{2}}\omega^n < \infty\}.
\end{align*} We will also need to consider sections of the canonical bundle $K_X$. Given a metric $e^{-u}$ on $L$ we will use the Hermitian metric $e^{-u}(\omega^{n})^{-1}$ on $K_X$, but we will often suppress the dependence of $\omega$ for notational clarity. For instance, we use the notation $$\n\sigma\n^2_{q\phi} = \frac{\sigma\wedge\bar\sigma}{\omega^n}e^{-q\phi}$$ for the norm. We then set
\begin{align*}\Ga^2_{qu}(K_X):= \{\sigma\in H^0(X,K_X)~\Big |~ \int_X ||\sigma||_{qu}^2\omega^n < \infty\}
\end{align*} Note that $||\sigma||_{qu}^2\omega^n  = \sigma\wedge\bar\sigma e^{-qu},$ and hence the integral in the second line above is independent of $\omega$ even though the point-wise norm of course depends on $\omega$. The aim of this section is to prove two kinds of B\'ezout estimates, one on $X$ (Proposition \ref{prop:Bezout}) while the other one on the projectivized tangent bundle $\PP(T_X)$ (Proposition \ref{prop:Bezout-proj}).
\subsection{Preliminary estimates}

\begin{lem}\label{lem:section-est-non-smt}
Let $(X^n,\omega)$ be a K\"ahler manifold with $Ric\geq 0$. Let $s \in \Ga^p_{q\psi}$. Then there exists a constant $C>0$ depending only on $p,q$, the Lipschitz constant of $\phi$ and $\int_X\n s\n_{q\psi}^p$ such that for all $0\leq t\leq 1$, $$\int_X \n s\n_{qu_t}^p,~\int_X\n s\n_{qu_t}^{p-2}||\partial_{qu_t} s||_{qu_t}^2 < C,$$ where $\partial_{qu_t} s := \partial s + qs\partial u_t$ is the $(1,0)$ part of the covariant derivative of $s$ with respect to the Chern connection of $e^{-qu_t}$.
\end{lem}

\begin{proof}
Note that the heat flow estimate above gives $e^{-qu_t} \leq Ce^{-q\psi}$. The estimate on $\n s\n_{qu_t}^p$ follows immediately from this estimate. By the Bochner identity, in the distributional sense we have $$\Delta \n s\n_{q\psi}^{p} \geq \frac{p^2}{4} \n s\n_{q\psi}^{(p-2)}||\partial_{q\psi} s||_{q\psi}^2 - \frac{Ap}{2}\n s\n_{q\psi}^{p},$$ where we assume that $\ddbar \psi\leq A\omega .$ Let $\chi_R$ be the family of cut-off functions from Lemma \ref{lem:cut-off}. Then multiplying the above inequality by $\chi_R$ (for $R\geq R_0)$ and integrating by parts $$\int_{B(o,\theta^{-1}R)}\n s\n_{q\psi}^{p-2} ||\partial_{q\psi} s||_{q\psi}^2 \leq \frac{4}{p^2}\int_X \n s\n_{q\psi}^p \Delta\chi_R + \frac{2A}{p}\int_X\n s\n_{q\psi}^p \leq C.$$ Letting $R\rightarrow \infty$ we the  have that $$\int_{X}\n s\n_{q\psi}^{p-2} ||\partial_{q\psi} s||_{q\psi}^2<C.$$ To have a similar estimate with the metric $e^{-qu_t}$ instead of $e^{-q\psi}$, we note that $$\partial_{qu_t}s = \partial_{q\psi} s + q\partial(u_t - \psi)\cdot s,$$ and so using the fact that $|\partial u_t|,|\partial \psi| \leq 1$ and also $e^{-qu_t} \leq Ce^{-q\psi}$ we have $$\n s\n_{qu_t}^{p-2}\n\partial_{qu_t}s\n_{qu_t}^2 \leq C \n s\n_{q\psi}^{p-2}\n\nabla_{q\psi} s\n_{q\psi}^2,$$ and the required estimate follows by integrating.

\end{proof}

We also need a slight variation of the above integral estimates, which follow from the fact that $u_t$ and $\psi$ are uniformly Lipschitz.

\begin{cor}
Let $s\in  \Gamma^p_{q\psi}$. Then there exists a constant $C>0$ such that for all $0\leq t\leq 1$,
\begin{gather}
\int \Vert ds \Vert_{qu_t}^2 \Vert s \Vert_{qu_t}^{p-2} \omega^2\leq \int \Vert ds \Vert_{q\psi}^2 \Vert s \Vert_{q\psi}^{p-2} \omega^2 \leq C.
\label{eq:integralofds}
\end{gather}
Hence (using Cauchy-Schwarz),
\begin{gather}
\int \Vert ds \Vert_{qu_t} \Vert s \Vert_{qu_t}^{p-1} \omega^2\leq \int \Vert ds \Vert_{q\psi} \Vert s \Vert_{q\psi}^{p-1} \omega^2 \leq C.
\end{gather}
\label{lem:integrability}
\end{cor}

A consequence of this estimate that we will need in the proof of B\'ezout estimates is the following.

\begin{lem}\label{lem:est-lpnorm-ddbarphi-omega}
Let $s\in \Ga^p_{q\psi}$. Then there exists a constant $C>0$ such that for all $0\leq t\leq 1$, \[\int_X \log(1 + \n s\n^p_{qu_t})\ddbar u_t\wedge \omega < C$$ and $$\int_X \n s\n^{p}_{qu_t}\ddbar u_t\wedge\omega< C.\]
\end{lem}
\begin{proof}
We prove the second estimate. The first one then follows from the inequality $\ln(1+x) \leq x.$ Let $\chi_R$ be the family of cut-off functions from Lemma \ref{lem:cut-off}, and let  $$I_t(R) := \int_X \chi_R  \n s \n^{p}_{qu_t}\ddbar u_t \wedge\omega.$$ The required estimate follows from the following claim: There exists a constant $C>0$ such that  $I_t(R) \leq C$ for all $0\leq t\leq 1$ and for all $R\geq R_0$, where $R_0$ is as in Lemma \ref{lem:cut-off}. First assume that $t>0$. Lemma \ref{lem:currents} implies that we can integrate by parts. Doing so, we see that $$ I_t(R) = \int_X\n s\n^p_{qu_t}\partial\chi_R\wedge\bar\partial u_t\wedge\omega + \int_X\chi_R\partial\n s\n_{q u_t}^p \wedge\dbar u_t\wedge\omega.$$ The first term is clearly uniformly bounded by Lemma \ref{lem:section-est-non-smt} and the uniform Lipschitz control on $u_t$. For the second term we note that if $\al,\be$ are two $(1,0)$ forms then $|\al\wedge\bar\be\wedge\omega| \leq |\al||\be|\omega^2,$ and since $u_t$ is Lipschitz with Lipschitz constant bounded by $A$, we then have that $$ \Big|\int_X\chi_R\partial\n s\n_{qu_t}^p \wedge\dbar u_t \wedge\omega\Big| \leq A\int_X\chi_R|\nabla \n s\n^p_{qu_t}| \omega^2.$$ It is easy to see (using Lemma \ref{lem:currents}) that  $$\int_X\chi_a|\nabla \n s\n^p_{qu_t}| \omega^2\leq C\left(\int_X\chi_R\n s\n^p_{qu_t} \omega^2 + \int_X\chi_R\n\nabla_{qu_t}s\n_{qu_t} \Vert s \Vert_{qu_t}^{p-1} \omega^2 \right).$$ Each term is uniformly bounded by the lemmata above. Finally, the claimed estimate at $t = 0$ follows from standard Bedford-Taylor theory (cf. \cite{Dem-book}) and the fact that $u(x,t)$ converges uniformly to $\phi(x)$ on compact sets.

\end{proof}

\begin{lem}\label{lem:poly-growth-sections}Let $(X^n,\omega)$ be a K\"ahler manifold with $BK_{\omega}>0$. Fix a $o\in X$. There exists a constant $C>0$ such that for any $x\in X$, $r>0$ and $t\in [0,1]$, we have the following:
\begin{enumerate}
\item For any $s\in \Ga^p_{q\psi}$,
\begin{align*}
\n s\n_{qu_t}^p(x) \leq C(1 + d(x,o)^{2n}).
\end{align*}
\item For any $s\in \Ga^p_{q\psi}$ and integer $N \geq p-1$,
$$\Vert ds^{N+1} \Vert_{qu_t}(x) = (N+1)\Vert ds\Vert_{qu_t} \Vert s \Vert_{qu_t}^{N} (x)\leq C (1 + d(x,o)^{2n})^{\frac{N+1}{p}}.$$
\item For any $\sigma\in \Ga^2_{q\psi}(K_X)$, $$\n \sigma \n_{qu_t}^2(x) \leq C(1+d(o,x)^{2n}).$$
\end{enumerate}
\end{lem}
Note that $ds$ is well defined since $L$ is the trivial line bundle. Indeed, $ds$ is a section (though not necessarily holomorphic) of $L\otimes T^*M$ and we equip the bundle with the family of metrics $e^{-qu_t}\omega^{-1}$.

\begin{proof} In each of the cases it suffices to obtain polynomial bounds for $\n\cdot\n_{q\psi}$ since $e^{-qu_t}\leq C e^{-q\psi}$.
\begin{enumerate}
\item We first prove this for the weight $u_1 = \psi$. By Lemma \ref{lem:currents}, since $\ddbar\psi \leq A\omega$ for some $A>0$, we have that $$\Delta\n s\n_{q\psi}^p \geq -A\n s\n_{q\psi}^p$$ holds in the distributional sense for some constant $A>0$. The Lemma is then proved by applying Corollary \ref{cor:poly-growth-general} with $f = \n s\n_{q\psi}^p$, and using the fact that $$\int_{B(x,1)}f = \int_{B(x,1)} |s|^p e^{-pq\psi/2} < C$$ for some uniform constant $C>0$.
\item The bundle $L\otimes T^*M$ with the metric $e^{-q\psi}\omega^{-1}$ has curvature bounded above by $A$ since $T^*M$ is Griffiths negative with the Hermitian metric $\omega^{-1}$, and so by Lemma \ref{lem:bochner-Weitzenbock-distributional}, we have that $$\Delta\n ds^{N+1}\n \geq -A \n ds^{N+1}\n,$$ and so by Corollary \ref{cor:poly-growth-general},
\begin{align*}
\n ds^{N+1}\n &\leq C(1+ d(x,o)^{2n})\int_{B(x,1)}\n ds\n\n s\n^{N}\\
&\leq C(1+d(x,o)^{2n})(1+ d(x,o)^{2n})^{\frac{N-p+1}{p}}\int_{B(x,1)}\n ds \n \n s\n^{p-1} \\
&\leq  C(1+d(x,o)^{2n})^{\frac{N+1}{p}}.
\end{align*}
Note that we used the first part to obtain the bound in the second line.
\item Since the curvature of the metric $(\omega^n)^{-1}$ is $Ric \leq 0$, once again by  Lemma \ref{lem:currents}, $$\Delta\n\sigma\n^2 \geq -A\n \sigma\n^2 $$ for some $A>0$. Once again by Corollary  \ref{cor:poly-growth-general} we obtain the desired polynomial bound.
\end{enumerate}
\end{proof}

\subsection{The main estimate}\label{subsec:mainestimate}  We now prove a general B\'ezout estimate. This next proposition is really the key technical input in the proof of the main theorem, and will be used several times. First we introduce some more notation. We let $$\Ga_{qu}:= \sum_{p\in(0,\infty)} \Ga^p_{qu},$$ where the right hand side denotes the vector space spanned by $\Ga^p_{qu}$ as $p$ varies over $(0, \infty)$. For any element $s\in \Ga_{qu}$ for the form $s = \sum_{i=1}^n s_i$, we set $$\tilde s := \prod_{i=1}^ns_i.$$ If $s_i\in \Ga^{p_i}_{qu}$, then clearly $\tilde s\in \Ga^{r}_{nqu}$ where $r^{-1} = \sum_{i=1}^n p_i^{-1}$. We also set $p(s) = \max_i p_i$.
\begin{prop}\label{prop:Bezout}
 Let $(X^2,\omega)$ be a K\"ahler surface satisfying $\sec_\omega > 0.$  Let $s_i\in \Gamma^{p_{i}}_{q\psi}$ and $t_j\in \Gamma^{a_j}_{b\psi}$ for $i\in\{1,\cdots,n\}$ and $j\in \{1,\cdots,m\}$. Let $p_{max} = \max_ip_i$ and $a_{max} = \max_j a_j$.  We set $$s = \sum_{i=1}^ns_i,~t = \sum_{j=1}^lt_j,~\tilde s = \prod_{i=1}^ns_i\text{ and }\tilde t = \prod_{j=1}^mt_j.$$ Finally, let $\sigma\in \Ga^2_{\la\psi}(K_X)$ and $\tau \in \Ga^2_{\mu\psi}(K_X)$. Now suppose that $v_1,v_2$ are non-negative continuous functions with the following properties:
  \begin{enumerate}
  \item There exists a constant $N_1\geq 17p_{max} $, $M_1\geq p_{max}/2$, and constants $A_1,B_1,C_1 \geq 0$ and $\ep_1>0$ such that $$v_1 \leq A_1 \log \Big(1 + \frac{\n s\n_{q\phi}^{2M_1}}{\ep_1^2}\Big)+ B_1\log\Big(1+ \frac{\Vert ds \Vert_{q\phi}^2 \n \tilde s \n_{q\phi}^{2N_1}}{\ep_1^2}) + C_1\log\Big(1+\frac{\Vert \sigma \Vert_{\la\phi}^{2}}{\ep_1^2}\Big).$$
  \item Similarly suppose there are constants  $N_2\geq 17a_{max}$, $M_2\geq a_{max}/2$ and constants $A_2,B_2,C_2\geq 0$ and $\ep_2>0$ such that $$v_2 \leq A_2 \log \Big(1 + \frac{\n t\n_{b\phi}^{2M_2}}{\ep_2^2}\Big)+ B_2\log\Big(1+ \frac{\Vert dt \Vert_{b\phi}^2 \n \tilde t \n_{b\phi}^{2N_2}}{\ep_2^2}) + C_2\log\Big(1+\frac{\Vert \tau \Vert^{2}_{\mu\phi}}{\ep_2^2}\Big)$$ where the norms are with respect to $\phi$.
  \item Finally suppose that there exist constants $\be_1,\be_2\geq 0$ such that $$\zeta_i:= \ddbar v_i + \be_i\ddbar\phi\geq 0.$$
  \end{enumerate}Then
  \begin{equation*}
\int_X \zeta_1\wedge\zeta_2 \leq \be_1\be_2\int_X(\ddbar\phi)^2 \leq C.
  \end{equation*}
where $C$ is independent of the constants $\ep_1,\ep_2$ and the only dependence on the sections is via $\be_1,\be_2$ ie. the positivity of $v_1$ and $v_2$. Here the integral is interpreted in the sense of Bedford-Taylor theory.
  \end{prop}

  \begin{lem}\label{lem:v-log-est}
Let $s\in \Ga^p_{q\psi}$, $\ep>0$ and $\al \in (0,1)$. Let $N+1\geq (4+\al)\al^{-1}p$. Then there exists constant $C>0$ independent of $R$ and $t\in [0,1]$ such that $$\log\Big(1 + \frac{2}{\ep^2}\n ds\n_{qu_t}^2 \n s\n_{qu_t}^{2N}\Big) \leq C (1+R)^\al\n s\n_{qu_t}^p$$ on $B(o,\theta^{-1}R)$.  In particular, if $N+1 \geq 17p$, then $$\log\Big(1 + \frac{2}{\ep^2}\n ds\n_{qu_t}^2 \n s\n_{qu_t}^{2N}\Big) \leq C (1+R)^{\frac{1}{4}}\n s\n_{qu_t}^p$$
  \end{lem}
\begin{proof}
We will use the following elementary inequality: For any $\delta\in (0,1]$, there exists a constant $C_\delta>0$ such that for all $x\in [0,\infty)$,
\begin{align}\label{eq:log-poly}
 \log(1+x)\leq C_\delta x^\delta.
 \end{align}
 Applying this, for a $\delta$ to be specified shortly, we have that there exists a $C>0$ such that
 \begin{align*}
  \log\Big(1 + \frac{2}{\ep^2}\n ds\n_{qu_t}^2\n s\n_{qu_t}^{2N}\Big)&\leq \frac{C}{\ep^{2\delta}} \n ds\n_{qu_t}^{2\delta}\n s\n_{qu_t}^{2N\delta}\\
  &=\frac{C}{\ep^{2\delta}} \Big(\n ds\n_{qu_t} \n s \n_{qu_t}^{N - \frac{p}{2\delta}}\Big)^{2\delta}\n s\n_{qu_t}^{p}\\
  &\leq \frac{C}{\ep^{2\delta}} (1+R)^{\frac{8\delta(N'+1)}{p}}\n s\n_{qu_t}^{p}
 \end{align*}
by Lemma \ref{lem:poly-growth-sections} as long as $$N' := N- \frac{p}{2\delta}\geq p-1.$$
Now let $$\delta = \frac{(4+\al)p}{8(N+1)}.$$ Then we verify that $$N'+1 = \frac{\al}{4+\al}(N+1) \geq p$$ since $N+1 \geq (4+\al)\al^{-1}p$. Furthermore, we also have that $$\frac{8\delta}{p}(N'+1) = \al.$$
\end{proof}

\begin{lem}\label{aux}

There exists $C>0$ independent of $R$ such that
$$ v_1 \leq C(1+R)^{1/4}\sum_i \n s_i\n^{p_i} + C\n \sigma\n^2$$

 and

 $$v_2 < C \sqrt{1+R}$$ on $B(o,\theta^{-1}R).$

\end{lem}

\begin{proof}

 We use Lemma \ref{lem:v-log-est}, \eqref{eq:log-poly} and the following elementary estimates:
 For any finite set $c_1,\cdots,c_n$ of non-negative real numbers and $M>0$ we have
  \begin{align}
 (\sum_i c_i)^{2M} \leq n^{2M-1}\sum_i c^{2M}_{i},\label{eq:elem-ineq-1}\\
 \log(1 + \sum_i c_i) \leq \sum_{i}\log(1+c_i)\label{eq:elem-ineq-2} \text{ and }\\
 \log(1+\prod_{i}c_i)\leq \sum_i\log(1+c_i).\label{eq:elem-ineq-3}
 \end{align}
By our hypothesis, $$v_1\leq A_1 f + B_1 g + C_1h,$$ where
 \begin{align*}
 f &=  \log \Big(1 + \frac{\n s\n_{q\phi}^{2M_1}}{\ep_1^2}\Big)\\
 g &= \log\Big(1+ \frac{\Vert ds \Vert_{q\phi}^2 \n \tilde s \n_{q\phi}^{2N_1}}{\ep_1^2}\Big) \text{ and }\\
 h&= \log\Big(1+\frac{\Vert \sigma \Vert_{\la\phi}^{2}}{\ep_1^2}\Big).
 \end{align*}
 By \eqref{eq:elem-ineq-1} and  \eqref{eq:elem-ineq-2} above,

 \begin{align*}
 f&\leq \log\Big(1 + \frac{n^{2M_1-1}}{\ep_1^2}\sum_i\n s_i\n^{2M_1}_{q\phi}\Big)\\
 &\leq \sum_i\log\Big(1+ \frac{n^{2M_1-1}}{\ep_1^2}\n s_i\n^{2M_1}_{q\phi}\Big)\\
 &\leq C\sum_i\n s_i\n^{p_i},
 \end{align*}
 where in the final line we apply \eqref{eq:log-poly} with $\delta = p_i/2M_1\in (0,1]$ for the $i^{th}$ term, and $C$ depends on $\ep_1$ but is independent of $R$. Trivially we also have that $$h \leq \frac{C}{\ep_1^2}\n \sigma\n^2.$$ Finally, applying all of \eqref{eq:elem-ineq-1},  \eqref{eq:elem-ineq-2} and  \eqref{eq:elem-ineq-3}, we have that
 \begin{align*}
 g&\leq \log\Big( 1+ \frac{2}{\ep_1^2}\n \tilde s\n^{2N_1}\sum_i\n ds_i\n^2\Big)\\
 &\leq \sum_i\log\Big( 1 + \frac{2}{\ep_1^2}\n ds_i\n^2 \n \tilde s\n^{2N_1}\Big)\\
  &\leq \sum_i \Big( \log\Big(1 + \frac{2}{\ep_1^2}\n ds_i\n^2\n s_i\n^{2N_1}\Big) + \sum_{j\neq i} \log(1+ \n s_j\n^{2N_1}) \Big).
 \end{align*}
 By Lemma \ref{lem:v-log-est} we then have that
 \begin{align*}
 g &\leq \sum_i \Big( C (1+R)^{1/4} \n s_i\n^{p_i} + \sum_{j\neq i} \log(1+ \n s_j\n^{2N_1}) \Big)\\
 &\leq C (1+R)^{1/4}\sum_i \n s_i\n^{p_i}.
 \end{align*}
for a possible bigger constant $C$ in the second line. Note that we once again applied the elementary inequality \eqref{eq:log-poly} in the second line. Putting everything together we see that there exist constants $C>0$ such that

 \begin{align*}
 v_1 \leq C(1+R)^{1/4}\sum_i \n s_i\n^{p_i} + C\n \sigma\n^2
 \end{align*}
  The proof of the estimate for $v_2$ is similar, and in fact easier, than that for $v_1$.  In the above estimate for $v_1$, the choice of $\delta$  was delicate  since we needed to peel off terms of the form $\n s_i\n^{p_i}$ (cf. proof of Lemma \ref{lem:v-log-est}). On the other hand, for the estimate for $v_2$ we have no such requirement and hence the above estimate easily follows from the polynomial growth of sections (Lemma \ref{lem:poly-growth-sections}) and the elementary estimates \eqref{eq:log-poly} - \eqref{eq:elem-ineq-3}.
\end{proof}

 \begin{proof}[Proof of Proposition \ref{prop:Bezout}] Let $\chi_R$ be a cut-off function as in Lemma \ref{lem:cut-off}, and let $$I(R):= \int_{X}\chi^3_R \zeta_1\wedge\zeta_2.$$  It is enough to show that
 $$\lim_{R\rightarrow\infty}I(R) \leq \be_1\be_2\int_X(\ddbar\phi)^2.$$Now $\zeta_1 = \ddbar(v_1+\be_1\phi)$ and so by Bedford-Taylor theory since $\zeta_1$ and $\ddbar\phi$ are both positive $(1,1)$ currents,
 \begin{align*}
 I(R) &= \int_X (v_1+\be_1\phi) \ddbar\chi^3_R \wedge\zeta_2\\
 &=\int_X v_1\ddbar\chi^3_R\wedge\zeta_2 + \be_1\int_X\chi_R^3\ddbar\phi\wedge\zeta_2\\
 &= \int_X v_1\ddbar\chi^3_R\wedge\zeta_2 + \be_1\int_X\chi_R^3\ddbar\phi\wedge \ddbar(v_2+\be_2\phi)\\
 &=\int_X v_1\ddbar\chi^3_R\wedge\zeta_2 + \be_1\int_X(v_2+\be_2\phi)\ddbar\chi_R^3\wedge\ddbar\phi\\
 &= \int_X v_1\ddbar\chi^3_R\wedge\zeta_2 + \be_1 \int_Xv_2\ddbar\chi_R^3\wedge\ddbar\phi + \be_1\be_2\int_X\chi_R^3(\ddbar\phi)^2.
 \end{align*}
 We write $I(R) = I_1(R) + \be_1I_2(R) + \be_1\be_2I_3(R),$ where
 \begin{align*}
 I_1(R) &=  \int_X v_1\ddbar\chi^3_R\wedge\zeta_2\\
 I_2(R) &= \int_Xv_2\ddbar\chi_R^3\wedge\ddbar\phi  \text{ and } \\
 I_3(R) &= \int_X\chi_R^3(\ddbar\phi)^2.
 \end{align*}
Clearly $$\lim_{R\rightarrow\infty}I_3(R) = \int_X(\ddbar\phi)^2.$$ It is thus enough to prove that $$\lim_{R\rightarrow\infty}I_1(R)= \lim_{R \rightarrow \infty} I_2(R) = 0.$$
By Lemma \ref{aux}, $$v_1 \leq CR^{1/4}\sum_i \n s_i\n^{p_i} + C\n \sigma\n^2.$$ Therefore,
  \begin{align*}
  I_1(R) &\leq \frac{C}{R^{3/4}}\sum_i\int_X\chi_R\n s_i\n^{p_i}\omega\wedge \zeta_2 + \frac{C}{R}\int_X\chi_R \n \sigma\n^2 \omega\wedge \zeta_2\\
\end{align*}
 Next, we let $$J_{i}(R) = \int_X\chi_R \n s_i\n^{p_i}\omega\wedge\zeta_2. $$ We now consider the forms $\tilde\zeta_2 = \zeta_2 + \ddbar\psi = \ddbar \tilde h_2$, where $\tilde h_2 = v_2 + \be_2\phi + \psi$. This is a strictly positive form with continuous potential $\tilde h_2$. So by Richberg's approximation (cf. \cite[pg. 43]{Dem-book}) there is a sequence of strictly PSH functions $\tilde h_{2,k}$ such that $\tilde h_{2,k}\rightarrow \tilde h_2$ uniformly on compact sets, in particular on $B(o,\theta^{-1}R).$ We now let $$v_{2,k} = \tilde h_{2,k} - \be_2 u_{1/k} - \psi$$ where as usual $u_t$ is the heat flow solution with initial condition $\phi$. Clearly $v_{2,k}\rightarrow v_2$ uniformly on compact sets. By Bedford-Taylor theory (cf. \cite[pg. 147]{Dem-book}), $$\omega \wedge \ddbar \tilde h_{2,k} \rightarrow \omega\wedge(\zeta_2 + \ddbar\psi)\text{ and }\omega \wedge u_{1/k}\rightarrow \omega\wedge\ddbar\phi$$weakly as measures, and so
 \begin{align*}
 J_i(R) &= \lim_{k\rightarrow\infty}\int_X\chi_R\ s_i\n^{p_i}\omega \wedge\ddbar \tilde h_{2,k} - \int_X \chi_R \n s_i\n^{p_i}\omega\wedge\ddbar\psi\\
&= \lim_{k\rightarrow\infty} \int_X \chi_R \n s_i\n^{p_i}_{q\phi} \omega \wedge \ddbar v_{2,k} +\be_2 \int_X\chi_R \n s_i\n^{p_i}_{q\phi} \ddbar\phi\wedge  \omega.
 \end{align*}By Lemma \ref{lem:est-lpnorm-ddbarphi-omega}, the second integral above is bounded. We now let $$J_{i,k}(R) =  \int_X \chi_R \n s_i\n^{p_i}_{q\phi} \omega \wedge \ddbar v_{2,k}.$$ Note that since $J_i(R)$ is non-negative, to prove an upper bound on $J_1(R)$, it is enough to prove an upper bound on $J_{i,k}(R)$ which is independent of $k$. By approximating $\n s_i\n^{p_i}$ uniformly on the support of $\chi_R$ by smooth functions one can integrate by parts and write $$J_{i,k}(R) = \int_X\chi_R v_{2,k}\ddbar\n s_i\n^{p_i}\wedge \omega + 2\Re\int_X v_{2,k} \sqrt{-1}\partial \n s_i\n_{q\phi}^{p_i}\wedge \dbar \chi_R\wedge\omega + \int_X v_{2,k}\n s_i\n^{p_i}_{q\phi} \ddbar\chi_R\wedge\omega.$$

 Next, by Lemma \ref{aux},  there exists a constant $C>0$ independent of $R$ such that $$v_2 < C \sqrt{R}$$ on $B(o,\theta^{-1}R)$.  Since $v_{2,k}\rightarrow v_2$ uniformly on compact sets, for all $k>>1$ we also have that $$v_{2,k}\leq C \sqrt{R}$$ on $B(o,\theta^{-1}R)$. Now, since $|\ddbar\chi_R| \leq CR^{-1}$ it then easily follows that $$ \int_X v_{2,k}\n s_i\n^{p_i}_{q\phi} \ddbar\chi_R\wedge\omega = O\Big(\frac{1}{R^{3/4}}\Big).$$ For the second term above we use the estimate $|\nabla \chi_R| \leq A/R$ and obtain $$2\Re\int_X v_{2,k} \sqrt{-1}\partial \n s_i\n_{q\phi}^{p_i}\wedge \dbar \chi_R\wedge\omega  \leq \frac{C}{\sqrt{R}}\int_X \Big|\nabla \n s_i\n^{p_i}_{q\phi}\Big| \omega^2 = O\Big(\frac{1}{\sqrt{R}}\Big),$$ where the final integral on the right can be proved to be bounded by using the Kato inequality and Cauchy-Schwarz inequality, exactly as in the proof of Lemma \ref{lem:est-lpnorm-ddbarphi-omega}. Finally we need to control the first integral above in the expression of $J_{i,k}$. The issue is that $\ddbar \n s_i\n^{p_i}$ itself may not be positive. However clearly, $$ \int_X\chi_R v_{2,k}\ddbar\n s_i\n^{p_i}\wedge \omega  \leq  \int_X\chi_R v_{2,k}\Big(\ddbar\n s_i\n^{p_i} + \frac{p_iq}{2}\n s_i\n^{p_i}\ddbar\phi\Big)\wedge \omega.$$
 By Lemma \ref{lem:currents}, we have that $$\ddbar\n s_i\n^{p_i} + \frac{p_iq}{2}\n s_i\n^{p_i}\ddbar\phi\geq 0$$ and so
 \begin{align*}
 \int_X\chi_R v_{2,k}\ddbar\n s_i\n^{p_i}\wedge \omega  \leq C\sqrt{R}\int_X \chi_R \ddbar \n s_i\n^{p_i}_{q\phi}\wedge\omega + C\sqrt{R}\int_X \chi_R \n s_i\n^{p_i}\ddbar\phi\wedge\omega.
 \end{align*}
 The second integral above is bounded above by Lemma \ref{lem:est-lpnorm-ddbarphi-omega}.Integrating by parts and arguing as before it is easy to see that the first integral is $O(1/\sqrt{R})$, and so $$ \int_X\chi_R v_{2,k}\ddbar\n s_i\n^{p_i}\wedge \omega = O(\sqrt{R}),$$ where the constant is of course independent of $k$. Putting everything together we have that $J_i(R) = O(\sqrt{R})$ and so $$I_1(R) = O\Big(\frac{1}{R^{3/4}}\Big) + \frac{C}{R}\int_X\chi_R \n \sigma\n^2 \omega\wedge \zeta_2.$$ Arguing exactly as above one can again show that the second integral is also $O(\sqrt{R})$, and hence $$I_1(R) = O\Big(\frac{1}{R^{1/4}}\Big).$$ Finally we estimate $I_2(R)$. This follows along the same lines as the estimate of $I_1(R)$ as in fact easier. We therefore only outline the main steps. Firstly by integrating by parts and using the estimates on $|\nabla \chi_R|$ and $|\ddbar\chi_R|$ we have that $$I_2(R) \leq C\int_X \chi_R v_2 \ddbar\phi\wedge\omega.$$ Next, exactly as above, we get the following estimate: $$  I_2(R) \leq \frac{C}{R^{3/4}}\sum_j\int_X\chi_R\n t_j\n^{a_j}_{b_j\phi}\ddbar\phi\wedge \omega + \frac{C}{R}\int_X\chi_R \n \tau\n^2_{\mu \phi} \ddbar\phi\wedge \omega.$$ But the integrals on the right are all bounded above by constants independent of $R$ by Lemma  \ref{lem:est-lpnorm-ddbarphi-omega}, and so $$I_2(R) = O\Big(\frac{1}{R^{1/4}}\Big).$$

   \end{proof}

 An immediate corollary, which will be used repeatedly, is the following.

 \begin{cor}\label{cor:Bezout-ricci-bound}
Let $s\in \Ga^p_{q\psi}$ and $\sigma\in \Ga^2_{\la\psi}(K_X)$. Then for any $\ep,\delta>0$, $$\int_X\zeta_\ep(s,p,q\phi)\wedge \zeta_\delta(\sigma,2,\la\phi)\leq \frac{pq\la}{2}\int_X(\ddbar\phi)^2.$$
In particular, $$\int_X \zeta_\ep(s,p,q\phi) \wedge Ric < \infty.$$
 \end{cor}

 In applications, we shall need the following generalization of the main estimate, where we allow sums of functions that satisfy the hypothesis of Proposition \ref{prop:Bezout}. The proof follows along exactly the same lines as the proof of the above Proposition, and we skip it for brevity.

 \begin{prop}\label{prop:Bezout-gen}
 Let $v_1,v_2$ be continuous, non-negative functions on $X$ such that for each $i=1,2$, $$v_{i}\leq \sum_{j=1}^{n_i} v_{ij},$$ for some non-negative continuous functions $\{v_{ij}\}$ satisfying the following estimate: There exist sections $s_{ij}\in \Ga_{q_i\psi}$ and $\sigma_{ij}\in \Ga^2_{\la_{i}\phi}(K_X)$, and constants $\ep_{ij}>0$, $N_{ij}\geq 17p(s_{ij})$, $M_{ij}\geq p(s_{ij})/2$ and $A_{ij},B_{ij},C_{ij}\geq 0$ such that $$v_{ij}\leq A_{ij}\log\Big(1 + \frac{\n s_{ij}\n^{2M_{ij}}}{\ep_{ij}^2}\Big)+ B_{ij}\log\Big(1 + \frac{\n ds_{ij}\n^2\n \tilde s_{ij}\n^{2N_{ij}}}{\ep_{ij}^2}\Big) + C_{ij}\log\Big(1+ \frac{\n \sigma_{ij}\n^2}{\ep_{ij}^2}\Big).$$
 Suppose further that for some $\be_i>0$, $$\ddbar v_i + \be_i\ddbar\phi \geq 0.$$ Let $\zeta_i = \ddbar v_i + \be_i\ddbar\phi$. Then $$\int_X \zeta_1\wedge\zeta_2  \leq \be_1\be_2\int_X(\ddbar\phi)^2.$$
 \end{prop}

 \subsection{B\'ezout estimates on $\PP(T_X)$}\label{sec:Bezout-proj}  Following the strategy of Mok, we will also need B\'ezout estimates on the projectivised tangent bundle. Let $(X^2,\omega)$ be a K\"ahler surface with $\sec_\omega>0$. Let $T_X$ denote the holomorphic tangent bundle of $X$ and let $M=\PP(T_X)$ be its projectivisation. Then $M$ carries a tautological line bundle $T = O_{\PP(T_X)}(-1)$ whose fibre over a point $(x,[v])\in M$ is simply the complex line spanned by $v$. We denote its dual by $T^*$. The metric $\omega$ induces a fibre-wise  Hermitian metric $\theta$ on $T^*$ such that the first Chern form $$\xi:= -\ddbar\log\theta$$ restricts to the standard Fubini-Study metric on each $\PP^1$ fiber with the normalization chosen so that $$\int_{\PP^1}\xi = 2\pi.$$ We consider the following form $$\nu := \xi + 2\pi^*Ric.$$  Then by some calculations in \cite{Mok89} and \cite{CZ1} it follows that $\nu$ is a K\'ahler form on $\PP (T_X)$. Next, note that any holomorphic one-form $\al$ on $X$ corresponds to a section of $T^*$ which we denote by $\mu_\al$. In coordinates, if $\al = f_idz^i$, then $\mu_\al = f_iz^i$, where we interpret $z^i$ as hyperplane sections on $\PP(T_X)$. Suppose we now endow $T^*$ with the metric $h = e^{-q\phi}\theta$, then an easy application of Cauchy Schwarz gives the following elementary estimate:
 \begin{equation}\label{eq:est-mu-sec}
 ||\mu_\al||_h^2 \leq e^{-q\phi}\pi^*||\al||_\omega^2.
 \end{equation}

 \begin{prop}\label{prop:Bezout-proj}
 Let $s_i \in \Ga^{p_i}_{q\psi}$ for $i=1,\cdots,n$, and set $$s = \sum_{i=1}^n s_i \text{ and }\tilde s = \prod_{i=1}^ns_i.$$ Let $\mu(s)$ be the section of $T^*:= \mathcal{O}_{\PP(T_X)}(1)$ corresponding to $\tilde s^Nds$ for some $N$. Let $p= \max_i p_i$,
 \begin{align*}
 \zeta_\ep(s):= \Big(\ddbar \log(\Vert s\Vert ^p + \ep^2) + \frac{pq}{2}\ddbar\phi\Big) \text{ and } \\
 \eta_\delta(s):= \ddbar\log(\delta^2 + \n \mu \n^2) + \nu + (N+1)q\ddbar\pi^*\phi,
 \end{align*}
where the norms as usual are calculated with weights $e^{-q\phi}$. If $N \geq 17p$, then $$\int_{\PP(T_X)}\pi^*\zeta_\ep(s)\wedge \eta_\delta(s) \wedge \nu < \infty.$$
 \end{prop}
 \begin{proof}
 Let $\chi_R$ be the family of cut-off functions as  before (cf. Lemma \ref{lem:cut-off}). It is enough to show that $$I(R) := \int_{\PP(T_X)}\pi^*\chi_R^3\pi^*\zeta_\ep(s)\wedge \eta_\delta(s) \wedge \nu$$ is uniformly bounded. We write $I = I_1+(N+1)qI_2+I_3$ where
 \begin{align*}
 I_1(R)&:= \int_{\PP(T_X)}\pi^*\chi_R^3\pi^*\zeta_\ep(s) \wedge\nu^2\\
 I_2(R) &:= \int_{\PP(T_X)}\pi^*\chi_R^3\pi^*\zeta_\ep\wedge\pi^*\ddbar\phi\wedge\nu \text{ and }\\
 I_3(R)&:= \int_{\PP(T_X)}\pi^*\chi_R\pi^* \zeta_\ep(s) \wedge  \ddbar\log(\delta^2 + \n \mu \n^2) \wedge\nu.
 \end{align*}
To estimate the integrals we use the following basic pushforward formulae: For any closed $(1,1)$ and $(2,2)$ forms $\al$ and $\be$ on $X$, $$\int_{\PP(T_X)}\pi^*\al \wedge \xi^2 =-2\pi\int_X \al \wedge Ric  \text{ and } \int_{\PP(T_X)}\pi^*\be \wedge\xi = 2\pi\int_X\be.$$
Then
\begin{align*}
I_1(R) &= \int_X \pi^*\chi_R^3\pi^*\zeta_\ep(s) \Big( \xi^2  +4\xi \wedge\pi^*Ric\Big)\\
&=  6\pi\int_X\chi_R^3\zeta_{\ep}(s)\wedge Ric.
\end{align*}  The integral above is then finite by Corollary \ref{cor:Bezout-ricci-bound}. Next, $$I_2(R) = \int_X\chi_R^3 \zeta_\ep(s)\wedge\ddbar \phi \leq C,$$ where the bound follows from the arguments in the proof of Proposition \ref{prop:Bezout}. Finally, integrating by parts (by Bedford Taylor theory) and using the elementary estimate \eqref{eq:est-mu-sec} we have
\begin{align*}
I_3(R) &\leq \frac{C}{R}\int_{\PP(T_X)}\pi^*\log(1 + \frac{\n \tilde s\n^{2N}\n ds \n^2}{\ep^2})\pi^*\chi_R\pi^*\omega\wedge \pi^*\zeta_\vep(s) \wedge\nu\\
&=\frac{C}{R}\int_{X}\chi_R \log(1 + \frac{\n \tilde s\n^{2N}\n ds \n^2}{\ep^2})\zeta_\vep(s)\wedge\omega \xrightarrow{R\rightarrow \infty}0,
\end{align*}
once again by the arguments in the proof of Proposition \ref{prop:Bezout}.

\end{proof}

 \section{Consequences of the B\'ezout estimates} \label{sec:Bezout-consequences} For the rest of the paper we assume that $(X^2,\omega)$ is a K\"ahler surface with positive sectional curvature. In this section we note some consequences of the B\'ezout estimates.
 \subsection{Multiplicity estimate}

 Recall that the Lelong number of a closed, positive $(1,1)$ current $T$ at a point $o\in X$ is defined by $$\nu(T,o) = \lim_{r\rightarrow 0}\frac{1}{r^{2n-2}}\int_{B_{\omega_{\CC^n}}(o,r)}T\wedge \omega_{\CC^n}^{n-1},$$ where $\omega_{\CC^n}$ is the Euclidean metric in some local holomorphic coordinates. For any holomorphic section $s$ the multiplicity $\mult_{x}(s)$ of $s$ at $x\in X$ is given by $$\mult_x(s) = \nu\Big(\frac{\sqrt{-1}}{2\pi}\partial\overline{\partial}\log|s|^2;x\Big),$$ where $\nu(T,x)$ denotes the Lelong number of a positive current $T$ at the point $x$.

\begin{prop}\label{prop:mult-est}
Given any $o\in X$, there exists $C=C(X,o)$ such that the following holds: For any non-trivial section $s = \sum_{i=1}^n s_i,$ where each $s_i\in \Ga^{p_i}_{q\psi}$ we have
$$\mult_{o}(s) \le C q.$$
\end{prop}
Note that crucially $C$ is independent of $s,n,p_i$ and of course $q$.
\begin{proof}
 Let $B(o,1)$ be a coordinate neighbourhood around the point $o$ and let $p = \max_i\{p_i\}$. Then by monotonicity of Lelong numbers,
 \begin{equation}\label{eq:mult-est}
 \mult_o(s) \leq \frac{1}{\pi p}\int_{B(o,1)}\ddbar\log|s|^p\wedge\omega_{Euc}\leq \frac{C}{p\pi}\int_{B(o,1)}\ddbar\log|s|^p\wedge Ric,
 \end{equation}
for some constant $C>0$ that depends only on $o$ and $\omega$. Note that we are using the fact that $\Rc$ is strictly positive. By Lemma \ref{lem:currents}, for any $\ep>0$, $$\ddbar \log(\n s\n^{p} + \ep^2) + \frac{pq}{2}\ddbar\phi\geq 0.$$  To estimate the integral we use Proposition \ref{prop:Bezout}. First, by Proposition \ref{prop:section-existence-hormander}, there exists $\la>>1$ and $\sigma\in H^0(X,K_X)$ such that $$\int_X \sigma\wedge\bar\sigma ~e^{-\la\psi} < \infty.$$ We fix this $\sigma$ once and for all. In particular, the $\sigma$ does not depend on $s$. By Corollary \ref{lem:ric-zeta-upperbound} we have that $$\zeta_\delta(\sigma,2,\la\phi) \geq \frac{\n \sigma\n_{\la\phi}^2}{\n \sigma\n_{\la\phi}^2 + \delta^2}Ric.$$
Multiplying by the cut-off function $\chi_R$ for $R>>1$, by Fatou's lemma
\begin{align*}
 \frac{C}{p\pi}\int_{B(o,1)}\ddbar\log|s|^p\wedge Ric &\leq   \frac{C}{p\pi}\int_{X}\chi_R\ddbar\log|s|^p\wedge Ric \\
 &=  \frac{C}{p\pi}\lim_{\ep\rightarrow 0}\int_{X}\chi_R\zeta_\ep(s,p,q\phi)\wedge Ric \\
 &\leq \frac{C}{p\pi}\lim_{\ep\rightarrow 0}\liminf_{\delta\rightarrow 0}\int_{X}\chi_R\zeta_\ep(s,p,q\phi)\wedge\zeta_\delta(\sigma,2,\la\phi)\\
 &\leq \frac{C}{p\pi}\lim_{\ep\rightarrow 0}\liminf_{\delta\rightarrow 0}\int_{X}\zeta_\ep(s,p,q\phi)\wedge\zeta_\delta(\sigma,2,\la\phi)\\
 &\leq Cq,
\end{align*} for some possibly larger constant $C$, where we applied Proposition \ref{prop:Bezout} with $v_1 =\log(\n s\n^{p}_{q \psi} + \ep^2)$, $v_2 = \log (\n \sigma \n^2_{\lambda \phi}+ \delta^2)$, $\be_1 = pq/2$ and $\be_2 = \la$. \end{proof}

 \subsection{Finiteness of a Gauss-Bonnet integral}

 \begin{prop}\label{prop:finiteGB}
 Let $s = \sum_{i=1}^n s_i$ where $s_i\in \Ga^{p_i}_{q\psi}$, and let $\Omega\subset X$ be an open subset such that $S = \{s = 0\}\cap \Omega$ is a smooth Riemann surface. Then  $$\int_S |K|\,dA < \infty.$$
 \end{prop}
 \begin{proof}
 Let $K^{\pm}:= \max(\pm K, 0).$ We need to prove that the integrals of both $K^+$ and $K^-$ are bounded. Since $TS$ is a subbundle of $TM$ and the curvature decreases in subbundles, we have $$\int_S K^+\,dA \leq \int_S Ric.$$ The integral on the right is bounded above by the B\'ezout estimates. Indeed, by Proposition \ref{prop:section-existence-hormander} there exists a $\sigma\in \Ga^2_{\la\psi}(K_X)$ for some $\la>>1$. By Corollary \ref{lem:ric-zeta-upperbound} $$\zeta_\delta(\sigma,2,\la \phi) \geq \frac{\n\sigma\n^2_{\la\phi}}{\delta^2 + \n\sigma\n^2_{\la\phi}}Ric.$$ So by the Poincar\'e-Lelong and Fatou's Lemma we have
  \begin{align*}
 \int_S Ric &= \lim_{\ep\rightarrow 0}\frac{1}{\pi p}\int_X\zeta_\ep(s,p,q\phi) \wedge Ric ~\text{ (where $p = \max_i p_i$)}\\
 &\leq  \lim_{\ep\rightarrow 0}\liminf_{\delta\rightarrow 0}\frac{1}{\pi p}\int_X \zeta_\ep(s,p,q\phi) \wedge \zeta_\delta(\sigma,2,\la\phi)
 \end{align*} The last integral is of course finite by Proposition \ref{prop:Bezout}. To obtain a lower bound on the integral of the Gauss curvature we use the idea of Mok to transfer the estimate to an integral on the projectivized tangent bundle $\pi:\PP(T_X)\rightarrow X$. We lift $S$ to $\hat S \subset \PP(T_X)$ by the assignment $S\ni x\mapsto (x,[T_xS]).$ Recall that holomorphic one-forms on $X$ correspond to holomorphic sections of $\mathcal{O}_{\PP(T_X)}(1)$. Let $\nu$ be the K\"ahler form on $\PP(T_X)$ defined in Section \ref{sec:Bezout-proj}. By a calculation in Mok, $$\int_S K\,dA = 2\pi\Big(2\int_S Ric - \int_{\hat S}\nu\Big)\geq -2\int_{\hat S}\nu.$$ The integral on the right is simply the area $\Ar(\hat S,\nu)$ of $\hat S$ with respect to the K\"ahler form $\nu$. Let $\hat\mu$ be the holomorphic section corresponding to the one-form $\hat\al = ds$. Then $\hat S = \pi^{-1}(S)\cap\{\hat\mu = 0\}.$ Let $\mu$ be the section corresponding to the holomorphic form $\al = \tilde s^N ds$ for some large $N$, where as usual, $\tilde s = \prod_{i=1}^ns_i.$ We let
 \begin{align*}
 \hat\eta_\delta(s) &=  \ddbar\log(\n \hat \mu \n^2 + \delta^2) + \nu + q\pi^*\ddbar\phi\text{ and }\\
 \eta_\delta(s) &= \ddbar\log(\n \mu \n^2 + \delta^2) + \nu + (N+1)q\pi^*\ddbar\phi,
 \end{align*} and define $\hat\eta_\delta(s)$ similarly by replacing $\mu$ by $\hat\mu$. Note that $\eta_\delta(s)$ and $\hat\eta_\delta(s)$ are positive $(1,1)$ forms on $\PP(T_X)$. We also let $\hat\eta_0(s),\eta_0(s)$ denote the respective currents of integration of the divisors $\hat\mu = 0$ and $\mu = 0$. Note that $\hat\eta_0(s)\leq \eta_0(s)$ as currents. We also let $$\zeta_\ep(s) = \ddbar\log(\n s\n^p + \ep^2) + \frac{pq}{2}\ddbar\phi,$$ where $p = \max_i p_i$. Then,
 \begin{align*}
 \Ar(\hat S,\nu) &=\lim_{R\rightarrow\infty}\int_{\hat S}(\pi^*\chi_R^3)~ \nu \\
& = \lim_{R\rightarrow\infty}\int_{\pi^{-1}(S)}(\pi^*\chi_R^3)\hat\eta_0(s)\wedge \nu\\
& \leq  \lim_{R\rightarrow\infty}\int_{\pi^{-1}(S)}(\pi^*\chi_R^3)\eta_0(s)\wedge \nu \\
&=  \lim_{R\rightarrow\infty}\lim_{\delta\rightarrow 0}\int_{\pi^{-1}(S)}(\pi^*\chi_R^3)\eta_\delta(s)\wedge \nu\\
&=\frac{1}{p} \lim_{R\rightarrow\infty}\lim_{\delta\rightarrow 0}\lim_{\ep\rightarrow 0}\int_{\PP(T_X)}(\pi^*\chi_R^3)\pi^*\zeta_{\ep}(s)\wedge \eta_\delta(s)\wedge \nu.
 \end{align*}
 By Proposition \ref{prop:Bezout-proj} this final integral has a uniform bound independent of $R,\ep,\delta$ if we choose $N \geq 17p$, and hence the area of $\Ar(\hat S,\nu)$ is finite.

 \end{proof}

\section{Proof of the main theorem} \label{sec:proof}

 With the B\'ezout estimates, and their consequences, in hand, we now follow the argument in \cite{CZ1} closely \footnote{The argument in \cite{CZ1}  in turn closely mimics the argument by Mok \cite{M} on affine embeddings.}. For the rest of the paper we let $(X,\omega)$ be a K\"ahler surface with positive sectional curvature. For $k,q\in \mathbb{Z}_{+}$, let $R_{q}^k:= \Ga_{q\psi}^{\frac{2}{k}}$,

$$R_q=\mathrm{sp}_{\RR}\displaystyle \bigcup_{k\in \mathbb{Z}_+} R^k_q\text{ and }R :=\bigoplus_{q\in \ZZ_+} R_q.$$
We note that $R$ is a graded algebra with the usual multiplication of functions. Indeed, by H\"older's inequality,  $$s \in R_{q}^k, \ s' \in R_{q'}^{k'} \implies ss' \in R_{q+q'}^{k+k'},$$ and hence in particular if $s\in R_q$ and $s'\in R_{q'}$ then $ss'\in R_{q+q'}.$ We let $$M:= \Bigl \{\frac{s}{t}~\Big| ~ s,t\in R_q\text{ for some } q \Bigr \}.$$  Then it is easy to check that $M$ is a subfield of the field of meromorphic functions on $X$.

\subsection{Finite generation}

The main goal of this section is to prove the following.
\begin{prop}[Siegel type result]\label{prop:siegel}
There exists a $q\in \ZZ_+$, sections $s_0,s_1,s_2\in R^1_q$ such that $$\Big[M:\CC\Big(\frac{s_1}{s_0},\frac{s_2}{s_0}\Big)\Big] < \infty.$$ Consequently, by the primitive element theorem there exists $\frac{s}{t}\in M$ such that $$M = \CC\Big(\frac{s_1}{s_0},\frac{s_2}{s_0},\frac{s}{t}\Big).$$
\end{prop}

We need the following consequence of the multiplicity estimate.

\begin{lem}\label{lem:finite-dimensionality}
There exists a constant $C>0$ such that $$ \dim_\CC R_q < Cq^2.$$
\end{lem}
\begin{proof} Recall that from the multiplicity estimate, we have a uniform constant $C>0$ such that for any $s\in R_q$, $$\mult_o(s) \leq Cq.$$
Now let $$m = \lfloor Cq\rfloor + 1\text{ and }d(m) = \dim_\CC(\mathcal{O}_m{(\CC^2)})\leq C'q^2,$$ and where $\mathcal{O}_m(\CC^2)$ is the space of polynomials in two variables of degree at most $m$. To prove the dimension estimate we consider the Poincar\'e-Siegel map: $P:R_q\rightarrow \CC^{d(m)}$ given by $$P(s) =  (s(o), Ds(o),D^2s_(o),\cdots, D^m(o)). $$ Clearly the map is linear. By the multiplicity estimate, the map $P$ is injective, and hence $$d_{q}:= \dim_\CC R_q  \leq d(m) \leq C'q^2.$$

\end{proof}

\begin{proof}[Proof of Proposition \ref{prop:siegel}]  The proof is standard (cf. \cite[pg. 384]{Mok90}), but we include it for the convenience of the reader. Fix a point $o\in X$. By an application of Proposition \ref{prop:section-existence-hormander}, we obtain a $q>>1$ and sections $s_0,s_1,s_2\in R_q^1$ such that  $\{s_1/s_0,s_2/s_0\}$ form holomorphic coordinates centred at $o$. We will prove that $M$ is a finite extension of $\CC(s_1/s_0,s_2/s_0)$ with the degree $$\Big[M:\CC\Big(\frac{s_1}{s_0},\frac{s_2}{s_0}\Big)\Big] \leq Cq^2$$ for some constant $C$. Let $f_1,\cdots,f_N \in M$ be linearly independent over $\CC(s_1/s_0,s_2/s_0)$. By multiplying out the denominators one can write $f_i = t_i/\tau$ for some $t_i,\tau\in R_\la$ and $i=1,\cdots,N$. For any positive integer $l$, any pair $\al:= (\al_1,\al_2)$ of indices with $|\al| = \al_1 + \al_2 \leq l$ and $i = 1,\cdots,N$ we let $$S_{i,\al} = s_1^{\al_1}s_2^{\al_2}s_0^{l - |\al|}t_i \in  R_{lq + \la}.$$ We claim that the collection $\{S_{i,\al}\}$ is linearly independent over $\CC$. Suppose not. Then there exist $c_{i,\al}\in \CC$ such that $$\sum_{i,\al}c_{i,\al}S_{i,\al} = 0.$$ Dividing by $\tau s_0^l$ we have that $$\sum_i f_i\Big(\sum_{\al}c_{i,\al}\Big(\frac{s_1}{s_0}\Big)^{\al_1}\cdot\Big(\frac{s_2}{s_0}\Big)^{\al_2}\Big) = 0.$$ By linear independence of $\{f_i\}$ and the fact that $s_1/s_0$ and $s_2/s_0$ are coordinates near $o$ we see that $c_{i,\al} = 0$ for all $i$ and $\al$. Since there are exactly ${l+2\choose 2}$ choices for the multi-index we have that $$N\frac{(l+2)(l+1)}{2} = N{l+2\choose 2} \leq  \dim_\CC R_{lq + \la} \leq C'(lq+\la)^2.$$ Letting $l\rightarrow\infty$ we see that $N\leq Cq^2$ for some $C>0$ independent which is universal and independent of $\la$ and the sections $\{f_i\}$.

\end{proof}

\subsection{Quasi-surjectivity} By clearing out denominators we may assume without loss of generality that $$M = \CC\Big(\frac{s_1}{s_0}, \frac{s_2}{s_0},\frac{s_3}{s_0}\Big).$$
Now, consider the rational map $F_0:X\dashrightarrow \PP^3$ defined by $$F_0(x) = [s_0(x):s_1(x):s_2(x):s_3(x)].$$ We first claim that the image lies inside a hypersurface in $\PP^3$. Indeed since $M$ is a finite extension over $\CC(s_1/s_0,s_2/s_0)$, the element $s_3/s_0$ is algebraic and hence satisfies a monic polynomial equation $$\Big(\frac{s_3}{s_0}\Big)^k + \sum_{j=0}^{k-1}R_j\Big(\frac{s_1}{s_0},\frac{s_2}{s_0}\Big)\Big(\frac{s_3}{s_0}\Big)^j = 0,$$ where for each $j$, $R_j$ is a rational function of two variables.  After clearing the denominators, one can see that the image $F_0(X)$ is contained in the zero set of a homogeneous polynomial on $\PP^3$. We denote the connected component containing $F(X)$ by $Z_0$.

Let $V$ be the union of $F^{-1}(Z_0^{sing})$ with the base and branching loci of $F$. Then $V$ is contained in the analytic subvariety of the form $V = \{f = 0\}$, where
\begin{equation}\label{eq:vanishing-V}
f := s_0^NQ(s_0,s_1,s_2,s_3)d\Big(\frac{s_1}{s_0}\Big)\wedge d\Big(\frac{s_2}{s_0}\Big),
\end{equation} where $Q$ is some polynomial corresponding to the singular set $Z_0^{sing}$ of $Z_0$, and $N$ is chosen large enough so that the section on the right is holomorphic. We let $U = X\setminus V$.
An elementary but key observation is the following.
\begin{lem}
$U$ is Stein.
\end{lem}
\begin{proof}
By the Gromoll-Meyer \cite{GM} theorem, $X$ is diffeomorphic to $\RR^4$, and since by Green-Wu \cite{GW} it is also Stein, $K_X$ is a (holomorphically) trivial line bundle. So $f$ is actually a holomorphic function. Finally, the complement of the zero-set of a holomorphic function on a Stein manifold is itself Stein by a theorem of Grauert-Remmert.  Alternately, one could also simply consider the function $$u = {\max}(\rho, \rho - \log|f|^2)$$ which gives a continuous strictly PSH exhaustion function on $U$. To obtain a smooth exhaustion function one can regularize using the method of Richberg \cite{Rich} (cf. \cite[Theorem 5.11]{Dem-book}).
\end{proof}
\begin{lem}\label{lem:bihol-U}
$F:U\rightarrow F(U)$ is a biholomorphism.
\end{lem}
\begin{proof}
Since $s_0 \neq 0$ on $U$ and $$d\Big(\frac{s_1}{s_0}\Big)\wedge d\Big(\frac{s_2}{s_0}\Big) \neq 0,$$ $s_1/s_0$ and $s_2/s_0$ form coordinates near any point on $U$. In particular, the map $dF$ is of full rank. By the inverse function theorem, $F$ is a local biholomorphism. To complete the proof we need to show that the map is injective. First note that $M$ separates points. Indeed given any $x,y\in U$, by Proposition \ref{prop:section-existence-hormander} one can find sections $t_1,t_2 \in R_{q}$ for some large $q>>1$ such that $$t_1(x),t_1(y)\neq 0, t_2(x) = 0 \text{ and }t_2(y) \neq 0.$$ Then $t_2/t_1\in M$ separates $x$ and $y$. Now suppose $F(x) = F(y)$. Then since $s_0(x),s_0(y)\neq 0$, we must have $$\frac{s_i(x)}{s_0(x)} = \frac{s_i(y)}{s_0(y)},i = 1,2,3.$$ But since $M$ is generated by polynomials in $s_1/s_0,s_2/s_0,s_3/s_0$, this contradicts the fact that $M$ separates points.
\end{proof}
Our goal now is to prove quasi-surjectivity ie. $Z_0\setminus F_0(U)$ is an algebraic set. By GAGA, it is enough to show that this is an analytic set. The main idea is to use Simha's criterion \cite{Simha}:

\begin{thm}[Simha]
Let $\Delta^2 = \{(z,w)~|~ z,w\in \Delta\} \subset \CC^2$ be a polydisc and let $D\subset \Delta^2$ be a domain of holomorphy. Suppose for each $a\in \Delta$, $$\#\Big(\{z = a\}\cap (\Delta^2 - D)\Big)< \infty.$$ Then $\Delta^2 - D$ is an analytic set.

\end{thm}
\begin{rem}
By following the arguments below carefully, one can show that in fact not only do the ``vertical slices" intersect the complement of $D$ in finite sets, but also, the number of points is locally bounded above. Therefore, instead of Simha's criterion, one can instead use a weaker criterion due to Narasimhan \cite{Nar} as in the proof of quasi-surjectivity in \cite{Mok89} (see also the remarks by Mok in \cite[pg.\ 425]{Mok89}).
\label{rem:AfterSimha}
\end{rem}

To apply Simha's criterion, we take a projective resolution of singularities $\pi:Z\rightarrow Z_0$. Then $F_0$ lifts to a birational map $F:X\dashrightarrow Z\subset \PP^N$ for some large $N$ such that once again $F:U\rightarrow F(U)$ is a biholomorphism. Note that the map is again given by sections in $M$ (in fact rational functions in $s_0,s_1,s_2,s_3$). The new goal then is to show that $Z\setminus F(U)$ is an algebraic subvariety.

  Let $\zeta_0 \in Z \setminus F(U)$ and let $\Delta^2$ be any coordinate polydisc in $Z$ centred around $\zeta_0$. Note that $D = \Delta^2 \cap F(U)$ is open. Moreover since $U$ is Stein and $F$ is a biholomorphism, $D$ is also Stein and hence a domain of holomorphy. Now, we claim that there exist linear homogeneous polynomials $T_0, T_1, T_2$ such that $T_0(\zeta_0)\neq 0$, $T_1(\zeta_0)=T_2(\zeta_0)=0$, $\frac{T_1}{T_0}, \frac{T_2}{T_0}$ produce a coordinate polydisc near $\zeta_0$, and the hyperplane $T_1=aT_0$ (for all $a$ small enough) intersects $F(U)$ non-trivially and $Z$ in a smooth connected (hence irreducible) curve. Indeed, consider any linear polynomial $T_0$ that does not vanish at $\zeta_0$. By Bertini's theorem the generic element of the linear system of hyperplanes passing through $\zeta_0$ is smooth away from $\zeta_0$. In fact, by genericity, it is smooth everywhere and is connected by the Lefschetz hyperplane theorem. Next, we consider any point in $F(U)$ and join it to $\zeta_0$ by means of a hyperplane.  By perturbing this hyperplane slightly and considering another generic hyperplane passing through $\zeta_0$ we arrive at the (defining) polynomials $T_1, T_2$. Let the corresponding sections in $R^1_q$ be denoted by $t_0, t_1, t_2$. The vertical slices in this coordinate system are given by $T_1 = \lambda T_0$ for $\lambda\in \CC$, or equivalently the vanishing of $T = T_1 - \lambda T_0$ (which for small $\lambda$ are still smooth and connected). Denote the corresponding sections in $R^1_q$ as $t=t_1-\lambda t_0$. So one has to prove that the hyperplane $T= 0$ (whose intersection $S_0$ with $Z$ is a smooth compact Riemann surface) intersects the complement of  $F(U)$ in $Z$ at most finitely many points. Let $\Omega_0$ be one of the connected components of  $S_0\cap F(U)$.  It is enough to show that  $S_0\setminus \Omega_0$ contains finitely many points. (Indeed, that would also prove that $\Omega_0$ is the only such component, and every vertical slice in  the small coordinate polydisc around $\zeta_0$ \emph{also} intersects $F(U)$. Therefore, Simha's criterion will be applicable.)
  Next, we need the following result.

 \begin{prop}\label{prop:huber} Let $S$ be a compact Riemann surface and $\Omega\subset S$ an open set with a complete K\"ahler metric $\eta$ such that $$\int_\Omega|K_\eta|\,dA_\eta < \infty,$$ and there exist a family of smooth compactly supported cutoff functions $\xi_a: \Omega\rightarrow [0,1]$ where $0<D\leq a <\infty$ such that
 \begin{enumerate}
 \item $\Vert \nabla \xi_a \Vert +\Vert \Delta \xi_a \Vert \leq \frac{C}{a}$
 \item For all $a<b$, $$\xi_a^{-1}(1) \subset \xi_b^{-1}(1)  \text{ and } \mathrm{Supp}(\xi_a)\subset \mathrm{Supp}(\xi_b)$$
 \item $\displaystyle \cup_{a\geq D} \xi_a^{-1}(1)=\Omega$.
 \end{enumerate}Then $S\setminus\Omega$ consists of finitely many points.

  \end{prop}

  This follows from Huber's theorem \cite{Hub}. In the present setting, one can provide a direct and self-contained proof using an elegant argument due to Mok \cite{Mok89}. We postpone the proof to the end of the section so as not to interrupt the flow of the rest of the proof.

  Continuing with the proof of quasi-surjectivity, we need to produce a complete metric $\eta$ on $\Omega_0 \subset S_0$ satisfying the Gaussian integral estimate, and also produce the required cut-off functions. Let $\Sigma$ be the closure of  $F\Big|_U^{-1}(\Omega_0)$ in $X$.

  \begin{lem}\label{lem:finite-closure}
  $\Sigma \setminus F\Big|_U^{-1}(\Omega_0)$ is a finite set.
  \end{lem}
  \begin{proof}
Note that $\Sigma \setminus F\Big|_U^{-1}(\Omega_0)$  is contained in $V_0:= \{t = 0\}\cap V$, where $t$ is chosen as above. Here $F$ and $V$ are as in (\ref{eq:vanishing-V}). It suffices to prove that $V_0$ is a finite set. Note that $\{t = 0\}$ is a smooth Riemann surface, and furthermore it contains points from $U$. In particular, $\{t= 0 \}$ and $V$ do not have any component in common. Therefore we can use Proposition  \ref{prop:intersectionnumberestimates} and our Bezout estimates to prove that the intersection is finite. To use the B\'ezout estimates we note that $V$ is contained in $\{\tilde f = 0\}$ where $$\tilde f = f\tilde s_0^M\prod_{i=1}^3 s_i^M\tilde s_i^{M},$$ and for a section $s\in R_q$ we use the notation $\tilde s$ as introduced in  Section \ref{subsec:mainestimate}. By Proposition \ref{prop:intersectionnumberestimates}, $$\#V_0\leq \lim_{(\ep_1,\ep_2)\rightarrow(0,0)}\int_X\tilde{\zeta}_{\ep_1}(t,2,q\phi)\wedge\tilde{\zeta}_{\ep_2}(\tilde f,2,\be_2\phi)$$ for some $\be_2>>1$ that we will specify shortly. Next, we note that $$\n ds_i \wedge ds_j\n^2 \leq \n ds_i\n^2 \n ds_j\n^2,$$ and so we have an estimate of the form
\begin{align*}
||\tilde f||^2_{\be_2\phi} \leq \sum_{i,j} \prod_{k=0}^3\n s_k\n^{2M^k_{ij}}\cdot\prod_{k\neq i,j} \n \tilde s_k\n^{2M}\cdot \n ds_i\n^2 \n \tilde s_i\n^{2M} \cdot \n ds_j\n^2\n\tilde s_j\n^{2M},
\end{align*}
where on the right we use norms for the appropriate degree. Now using the elementary inequalities \eqref{eq:elem-ineq-2} and \eqref{eq:elem-ineq-3}, if we set $v_2 = \log(1 + \ep_2^{-2}\n \tilde f\n^2)$, then we have that $v_2\leq \sum_{j=1}^{n_1} v_{2j}$ where each $v_{2j}$ satisfies the hypotheses in Proposition \ref{prop:Bezout-gen} for some choice of $\ep_{2j}>0$. Then by Proposition \ref{prop:Bezout-gen} applied to $v_2$ and $v_1 = \log(1+ \ep^{-2}_1\n t\n^2_{q\phi})$ we have that $\#V_0$ is finite.
  \end{proof}

From now on we identify $\Omega_0$ and $F^{-1}(\Omega_0)$. In general, $\Sigma$ may be singular. However, note that $\Sigma - \Omega_0$ is finite and hence $\Omega_0$ is an open subset of $\Sigma$. Since $\Omega_0$ is smooth, it then implies that $\Omega_0\subset \Sigma^{reg}$. Let $\pi:\Sigma'\rightarrow\Sigma$ be the normalization. Then $\Omega_0$ sits as an open subset of $\Sigma'$ and moreover, $\Sigma'\setminus \Omega_0$ is a finite collection of points, say $\{p_1,\cdots,p_m\}$. The restriction $\omega\Big|_{\Omega_0}$ defines an incomplete metric on $\Omega_0$. Consider neighbourhoods $D_i$ of $p_i$ isomorphic to discs $D(0,r_i) \subset \CC$. On each punctured disc $D_i^*$ one has a Poincar\'e-type hyperbolic metric of finite area with a complete end at $p_i$. Using partitions of unity one can glue these metrics to $\omega$ to obtain a complete metric $\eta$ on $\Omega_0\subset \Sigma'.$

We first produce the desired cutoff functions. Given the Poincar\'e hyperbolic metric $g$ on the punctured unit disc, let $s(p)$ be the hyperbolic distance of a point $p$ from the circle $S_{1/2}$ of Euclidean radius $\frac{1}{2}$. Note that $s(p)$ is smooth in a punctured disc of Euclidean radius $3/4$. (Indeed, by radial symmetry, $s(p)=d_{g}(\Vert p \Vert, \frac{3}{4})-d_g(\frac{1}{2},\frac{3}{4})$.) Let $\tilde{\chi}:\mathbb{R}\rightarrow [0,1]$ be a smooth function equal to $1$ on $(-\infty, 1]$ and $0$ on $[2,\infty)$. Define $\tilde{\xi}_a$ to be $1$ outside the punctured disc of Euclidean radius $\frac{1}{2}$ and $\tilde{\xi}_a(p)=\tilde\chi \left(\frac{s(p)}{a} \right)$ inside it. One can verify that indeed $\tilde{\xi}_a$ satisfies the desired derivative bound. Now define $\xi_a=\chi_a \Pi_{i=1}^m \tilde{\xi}_{a,p_i}$. This family of functions does the job. Finally,  we claim that $$\int_{\Omega_0}|K_\eta|<\infty.$$ Indeed, since the area of $\eta$ near $p_i$ is finite and the curvature near $p_i$ is constant, we see using Proposition \ref{prop:finiteGB} that the Gauss-Bonnet integral is finite. The finiteness of $S_0\setminus\Omega_0$ then follows from Proposition \ref{prop:huber}.

  \begin{proof}[Proof of Proposition \ref{prop:huber}]
Following Mok, we consider the space of integrable quadratic differentials on $\Omega$. This space is independent of $\eta$. Indeed, for any quadratic differential $q\in H^0(\Omega,K_\Omega^2)$ one can canonically (and independent of any metric) associate
a $(1,1)$ form $|q|$. For, suppose $q = f(z)dz^2$ locally, then we simply set $|q| = |f(z)|dz\wedge d\bar z$. One can easily check that this defines a global (positive) $(1,1)$. form. We then define $$\Ga^1(K_\Omega^2) = \{q\in H^0(\Omega, K_\Omega^2)~|~ \int_\Omega |q|<\infty\}.$$ We claim that $$\dim \Ga^1(K_\Omega^2) < \infty.$$ To prove this, by the Poincar\'e-Seigel argument, it is enough to prove a multiplicity estimate. As in the proof of Proposition \ref{prop:mult-est}, it in turn is enough to prove the following B\'ezout type estimate: For any $q\in \Ga^1(K_\Omega^2)$, $$\int_\Omega \ddbar\log(\ep^2+ \n q\n^2) - 2\int_\Omega K_\eta\,dA < C$$ for some $C$ independent of $\ep$.  Here the norm of $q$ is measured using the metric $\eta$. If $\eta = \sqrt{-1}h dz\wedge d\bar z$, then $$\n q\n^2 = |f(z)|^2h^{-2}.$$ $h^{-2}$ defines a Hermitian metric on $K_\Omega^2$ with curvature $-2K_\eta$. The second term above is bounded by hypothesis. To estimate the first term, fix a point $x_0 \in\Omega$ let $\xi_a$ be the family of cut-off functions as in the statement of the Proposition. Then
\begin{align*}
\int_\Omega \ddbar\log(\ep^2+ \n q\n^2) &= \lim_{a\rightarrow\infty}\int_{B_\eta(x_0,2a)}\xi^3_a \ddbar\log(1+ \frac{\n q\n^2}{\ep^2})
\end{align*}

We claim that the integral on the right converges to zero as $a\rightarrow\infty$. Indeed, integrating by parts, and using the elementary estimate $\log(1+x^2)\leq Cx$ for some $C>0$ and all $x>0$,
\begin{align*}
\int_{B_\eta(x_0,2a)}\xi^3_a \ddbar\log\left(1+\frac{ \n q\n^2}{\ep^2}\right) &\leq  \frac{C}{a}\int_{B_\eta(x_0,2a)}\xi_a \log\left(1+ \frac{\n q\n^2}{\ep^2}\right)\eta\\
&\leq   \frac{C}{a\ep}\int_{B_\eta(x_0,2a)}\xi_a\n q\n\eta\\
&\leq  \frac{C}{a\ep}\int_\Omega |q| \leq \frac{C'}{a\ep} \xrightarrow{a\rightarrow \infty}0.
\end{align*}

Now let $\{p_1,\cdots,p_k\} \in S\setminus\Omega.$ Following Mok's argument verbatim we can prove that $$k\leq \dim \Ga^1(K_\Omega^2) + 3-3g(S).$$ For the convenience of the reader, we outline the proof. The main idea is to construct a quadratic differential with at most simple poles at each $p_k$. The space of such quadratic differentials is $H^0(S,K_S^2\otimes \sO(p_i)\otimes \cdots\otimes\sO(p_k))$, where $\sO(p_j)$ denotes the line bundle corresponding to the divisor $p_j$. We denote the line bundle $K_S^2\otimes \sO(p_i)\otimes \cdots\otimes\sO(p_k))$ by $L_k$. Since $1/|z|$ is integrable near zero in the complex plane, clearly $H^0(S,L_k) \subset \Ga^1(K_\Omega^2)$. By Serre duality, $\dim H^1(S,L_k) = \dim H^0(S,L_k^*\otimes K_S)$. But $$\deg(L_k^*\otimes K_S) = -k - \deg(K_S) = - k + 2 - 2g(S)< 0$$ if $k > 2-2g(S)$ and hence $\dim H^1(S,L_k) = 0$ in this case. In particular this holds if $k\geq 3$.  Then by Riemann-Roch, $$\dim H^0(S,L_k) = \deg(L_k) - g(S) + 1 = 3g(S) - 3 + k,$$ and we get the required upper bound on $k$.
  \end{proof}


\subsection{Completion of the proof} We now complete the proof following the basic strategy adopted by Mok in \cite{Mok89}. To summarize the results from the previous section: we have a birational map $F_1:X\dashrightarrow Z_1 \subset \PP^{N_1}$ which is a biholomorphism on $U_1 = X\setminus V_1$, where $V_1$ is an analytic subvariety of $X$ defined by the vanishing of $$f_1 := Q_1(s_0,s_1,s_2,s_3)\sigma_1,$$ where $Q_1$ is homogeneous polynomial, say of degree $d_1$, and $\sigma$ is a holomorphic section of $K_X$ given by some $R_{2q}$-linear combination of $ds_i\wedge ds_j$. Moreover, $F(U_1)$ is Zariski dense in $Z_1$. We first observe that $V_1$ is a union of analytic curves and does not have any isolated points because $U_1$ is Stein (by construction), and so is $X$. Indeed, around an isolated point of $V_1$, there exists a coordinate ball $B$. The punctured ball is the intersection of $B$ with $U_1$ but it cannot be Stein by Hartog's phenomenon.

\begin{lem}
$V_1$ has finitely many irreducible components.
\end{lem}
\begin{proof}
This is essentially due to Demailly (cf. \cite{D}) and we repeat his argument. $Z_1$ and $Z_1\setminus F_1(U_1)$, being projective subvarieties, are of finite topological type, i.e., have finite-dimensional singular cohomology groups. Therefore $F_1(U_1)$, and hence $U_1$, are also of finite topological type.  Now, $X$ is diffeomorphic to $\RR^4$ by Gromoll-Meyer \cite{GM} and hence of finite topological type. By using relative cohomology, it then follows that $V_1$ is of finite topological type and hence must have finitely many irreducible components.
\end{proof}

Let $\{C_i\}_{i=1}^n$ be the irreducible curves in $V_1$ and let $x_i\in C_i$ be smooth points. Then there exist sections $t_0,t_1, t_2\in R_{b}$ for some $b >> 1$ such that for all $i$, $t_0(x_i) \neq 0$, $t_1(x_i) = t_2(x_i) = 0$ and $t_1/t_0,t_2/t_0$ form coordinates near $x_i$. By composing $F_1$ with a Veronese embedding, we may assume that $q$ is large enough so that we may choose $b = q$. Then by following the same argument as before we have that $$M = \CC\Big(\frac{t_1}{t_0},\frac{t_2}{t_0},\frac{t_3}{t_0}\Big),$$ for some $t_3\in R_{b}$. Again as before, we then have a birational map $F_2: X  \dashrightarrow \hat Z_2\subset \PP^{ N_2}$ defined over $U_2 = X\setminus V_2$ for some subvariety $V_2$ given again by the vanishing of $$f_2 := Q_2(t_0,t_1,t_2,t_3)\sigma_2,$$ where as before $Q_2$ is a homogeneous polynomial of degree $d_2$, and $\sigma_2$ is a $R_b$-linear combination of $dt_i \wedge dt_j$.   Moreover $F_2(U_2)$ is Zariski open in $Z_2$. Note that by construction $V_1$ and $V_2$ share no component. In particular the (set-theoretic) intersection $V_1\cap V_2$ contains isolated points.

\begin{lem}
$V_1\cap V_2$ is finite.
\end{lem}
\begin{proof} We can write $$f_1 = \sum_{i,j}Q_{ij}(s_0,s_1,s_2,s_3)ds_i\wedge ds_j,$$ where $Q_{ij}$ are homogeneous polynomials of common degree $d$ such that the exponent of $s_0$ in each term of $Q_{ij}$ is large, bigger than $M>>1$. We now modify $f_1$ in two ways: firstly we want the exponents of all $s_i$'s to be large, and we also need to multiply by large exponents of $\tilde s_i^M$ for all $i$. More specifically, for large enough $M$, we consider $$\tilde f_1 = f_1\tilde s_0^M\prod_{i=1}^3 s_i^M\tilde s_i^{M}.$$We define $\tilde f_2$ in a similar way. Clearly $V_i\subset \{\tilde f_i = 0\}$, and so by Proposition \ref{prop:intersectionnumberestimates} $$\# V_1\cap V_2 \leq \lim_{(\ep_1,\ep_2)\rightarrow(0,0)}\int_X \zeta_{\ep_1}(\tilde f_1, 2, \be_1\phi)\wedge\zeta_{\ep_2}(\tilde f_2, 2,\be_2\phi),$$ for some (explicitly computable) $\be_i>>1$. We now estimate the integral using our B\'ezout estimates. Firstly note that by Lemma \ref{lem:currents}, each of the $\zeta_{\ep_i}(\tilde f_i,2,q_i\phi) \geq 0$.  Next, exactly as in the proof of Lemma \ref{lem:finite-closure} we have the estimate
\begin{align*}
||\tilde f_1||^2_{\be_1\phi} \leq \sum_{i,j} \prod_{k=0}^3\n s_k\n^{2M^k_{ij}}\cdot\prod_{k\neq i,j} \n \tilde s_k\n^{2M}\cdot \n ds_i\n^2 \n \tilde s_i\n^{2M} \cdot \n ds_j\n^2\n\tilde s_j\n^{2M},
\end{align*}
where on the right we use norms for the appropriate degree. Once again, as in the proof of Lemma \ref{lem:finite-closure}, using the elementary inequalities \eqref{eq:elem-ineq-2} and \eqref{eq:elem-ineq-3}, if we set $v_1 = \log(1 + \ep_1^{-2}\n \tilde f_1\n^2)$, then we have that $v_1\leq \sum_{j=1}^{n_1} v_{1j}$ where each $v_{1j}$ satisfies the hypotheses in Proposition \ref{prop:Bezout-gen} for some choice of $\ep_{1j}>0$. Similarly we have the estimate that $$v_2:= \log\Big((1 + \frac{\n \tilde f_2\n^2}{\ep_2^2}\Big) \leq \sum_{k=1}^{n_2} v_{2k}$$ for some choice of functions $v_{2k}$ satisfying the hypothesis of the Proposition \ref{prop:Bezout-gen}. Then by Proposition \ref{prop:Bezout-gen} we have that $$\#V_1\cap V_2  \leq \be_1\be_2 \int_X(\ddbar\phi)^2 < \infty.$$
\end{proof}
Let $\{y_1,\cdots,y_m\}$  be the finitely many points in $V_1\cap V_2$. Possibly by increasing $q$ once again using Veronese embedding, we may assume that there exists a section $s\in R_q$ that does not vanish at any of these points. Moreover, sections in $R_q$ also separate tangent vectors at these points. By construction there is a point $x\in X$ and sections, which by an abuse of notation we once again label $S_0,S_1,S_2\in R_q$, such that $S_1/S_0,S_2/S_0$ form coordinates near $x$.  Then by the argument of
Proposition \ref{prop:siegel} we once again have that $$M = \CC\Big(\frac{S_1}{S_0},\frac{S_2}{S_0},\frac{S_3}{S_0}\Big).$$ Once again, possibly by raising $q$ and making sure that the denominators are the same we may further assume that all $S_0,S_1,S_2,S_3\in R_q$ ie. they all belong to the same degree $q$ piece. We now extend $\{S_0,S_1,S_2\}$ to a basis $\{s_0,\cdots,s_N\}$ of $R_q$, where $s_i = S_i$ for $i=0,1,2$,  and consider the corresponding Kodaira map: $\Phi:X\rightarrow \PP^N$ given by $$\Phi(x) = [s_0(x):s_1(x):\cdots:s_N(x)].$$ If $S_3$ is linearly independent of $S_1, S_2$, then we let $s_3=S_3$. In either case, $s_j/s_0$ is given by a rational function of $s_1/s_0, s_2/s_0, s_3/s_0$ for $j=4,\cdots,N$ (and for $j=3$ if $S_3$ is linearly dependent on $S_1, S_2$) and hence $\Phi(X)$ is contained in an algebraic subset of $\PP^N$. Since on a Zariski open subset of $X$, $\Phi$ is biholomorphic to its image, by connectedness, $\Phi(X)$ is contained in an irreducible $2$-dimensional component $Z$. We would like to show that $\Phi$ is injective following the argument in the proof of \cite[Proposition 3.3]{Mok90}. The argument in \cite{Mok90} does not directly apply since the image $\Phi(X)$ may not be contained completely in the smooth locus $Z_{reg}$. For instance, it may very well happen that $z= \Phi(x_1) = \Phi(x_2)\in Z\setminus Z_{reg}$ and that disjoint open neighbourhoods $U_{x_1}$ and $U_{x_2}$ around $x_1,x_2$ are mapped to two  different branches containing $z$ ie. $\Phi(U_{x_1}) \cap \Phi(U_{x_2}) = \{z\}$. In other words the germ $(Z,z)$ may not be irreducible. Note that it would be irreducible if $Z$ were normal. So to adapt the argument in \cite{Mok90}, we pass to the normalization \footnote{Passing to the resolution would not help since one would not get a lift of $\Phi$ from all of $X$ to the resolution}.

Now let $\mu:\tilde Z\subset \PP^{\tilde N}\rightarrow Z$ be the normalization and let $\tilde\Phi:X\rightarrow \tilde Z$ be the canonical lift defined by virtue of the universal property of normalizations. Since it will be useful later on, we say a few words about the construction of $\tilde Z$ and its embedding into $\PP^{\tilde N}$. Let
\[
A=\bigoplus_{r\ge 0}A_r=S/I(Z),\qquad S=\mathbb C[z_0,\dots,z_N],
\]
be the homogeneous coordinate ring of \(Z\). Since \(Z\) is irreducible,
\(A\) is a graded domain. Let \(\widetilde A\) be the integral closure of
\(A\) in its field of fractions. Choose \(m\ge 1\) such that the Veronese subring
\[
\widetilde A^{(m)}:=\bigoplus_{k\ge 0}\widetilde A_{km}
\]
is generated by its degree-one piece. Then
\[
\widetilde Z=\operatorname{Proj}(\widetilde A)
=\operatorname{Proj}(\widetilde A^{(m)}),
\]
and the embedding \(\widetilde Z\subset \mathbb P^{\widetilde N}\) is induced
by a basis of \(\widetilde A^{(m)}_1=\widetilde A_m\). \begin{lem}
$\tilde \Phi$ is an injective immersion.
 \end{lem}
 \begin{proof}
 We first prove that $\tilde\Phi$ is an immersion. Indeed, this follows immediately from the fact that  $$d\mu\circ d\tilde\Phi = d\Phi$$ and $\Phi$ is an immersion, where we interpret the equality as an equality of linear maps between the tangent space of $X$ and the Zariski tangent space of $Z$. Similarly $d\mu$ is the linear map between the Zariski tangent spaces of $\tilde Z$ and $Z$.

To prove injectivity, we first, we note that a consequence of Lemma \ref{lem:bihol-U} is that $\Phi\Big|_{U_1}$ is injective. This is because the injectivity on $U_1$ is preserved under the Veronese embeddings. This forces $\tilde\Phi\Big|_{U_1}$ to also be injective. Next, note (and this is the reason we needed to pass to the normalization) that $\tilde \Phi(X)$ is contained in the smooth locus of $\tilde Z$. This follows from the fact that $\tilde Z$ is normal, $\tilde\Phi$ is an immersion and that $\dim X = \dim \tilde Z$. Now suppose $z = \tilde\Phi(x_1) = \tilde\Phi(x_2)$. Then by the inverse function theorem there are disjoint open sets $W_i$ containing $x_i$ such that $\tilde \Phi(W_1) = \tilde\Phi(W_2)$ is an open neighbourhood of $z$. Since $U_1$ is dense in $X$, $U_1\cap W_i$ is non-empty for each $i=1,2$. But this contradicts the injectivity of $\tilde\Phi\Big|_{U_1}$.
  \end{proof}

Finally, to complete the proof we need to prove quasi-surjectivity again.

\begin{lem}\label{lem:qs-2}
$\tilde \Phi(X)$ is Zariski open in $\tilde Z$.
\end{lem}

To prove the Lemma we need to re-run the argument in the previous section using Simha's result. In order to do so, we will need the following elementary observation.

\begin{lem}\label{lem:algebro-geometric}
With the set-up as above, for each linear hyperplane $\tilde \xi = 0$ in $\PP^{\tilde N}$, there exists a homogeneous polynomial $Q \in \CC[z_0,\cdots,z_N]$ not identically vanishing on $Z$ such that $$\{\tilde\Phi^*\tilde\xi = 0\} \subset \{\Phi^*Q = 0\}.$$
\end{lem}

This is probably standard, but we include a proof for the sake of completeness. 

\begin{proof}

Any linear form
\(\widetilde \xi\) on \(\mathbb P^{\widetilde N}\) may be viewed as a nonzero
homogeneous element
\[
\widetilde \xi\in \widetilde A_m .
\]

Since \(\widetilde A\) is integral over \(A\), the element \(\widetilde \xi\)
is integral over \(A\). Hence there exists a monic relation
\[
\widetilde \xi^d+b_1\widetilde \xi^{d-1}+\cdots+b_d=0,
\qquad b_i\in A.
\]
Write each \(b_i\) as a finite sum of homogeneous pieces,
\[
b_i=\sum_{\nu} b_{i,\nu},\qquad b_{i,\nu}\in A_\nu .
\]
Since \(\deg(\widetilde \xi)=m\) in the original grading, taking the
homogeneous component of total degree \(dm\) in the above identity gives a
new monic relation
\[
\widetilde \xi^d+a_1\widetilde \xi^{d-1}+\cdots+a_d=0,
\qquad a_i:=b_{i,im}\in A_{im}.
\]
Equivalently, \(\widetilde \xi\) is integral over the Veronese subring
\[
A^{(m)}:=\bigoplus_{k\ge 0}A_{km}.
\]

Choose such a homogeneous monic relation with \(d\) minimal. Then
\(a_d\neq 0\). Indeed, if \(a_d=0\), then
\[
\widetilde \xi\bigl(\widetilde \xi^{d-1}+a_1\widetilde \xi^{d-2}
+\cdots+a_{d-1}\bigr)=0.
\]
Since \(\widetilde A\) is a domain and \(\widetilde \xi\neq 0\), this would
give a homogeneous monic relation of smaller degree, contradicting
minimality. We now prove that
\[
V_{\widetilde Z}(\widetilde \xi)\subset V_{\widetilde Z}(a_d).
\]
Let \(p\in V_{\widetilde Z}(\widetilde \xi)\). Choose
\(\widetilde \eta\in \widetilde A_m\) such that \(\widetilde \eta(p)\neq 0\).
On the affine chart \(D_+(\widetilde \eta)\), the relation above may be
divided by \(\widetilde \eta^d\), yielding
\[
\left(\frac{\widetilde \xi}{\widetilde \eta}\right)^d
+\frac{a_1}{\widetilde \eta}
\left(\frac{\widetilde \xi}{\widetilde \eta}\right)^{d-1}
+\cdots+
\frac{a_d}{\widetilde \eta^d}
=0.
\]
Here each \(a_i/\widetilde \eta^i\) is a regular function on
\(D_+(\widetilde \eta)\), because \(a_i\in A_{im}\). Since
\((\widetilde \xi/\widetilde \eta)(p)=0\), evaluating at \(p\) gives
\[
\left(\frac{a_d}{\widetilde \eta^d}\right)(p)=0,
\]
hence \(a_d(p)=0\). Therefore
\[
V_{\widetilde Z}(\widetilde \xi)\subset V_{\widetilde Z}(a_d).
\]
Now choose a homogeneous polynomial \(Q\in S_{dm}\) whose image in
\(A_{dm}=S_{dm}/I(Z)_{dm}\) is \(a_d\). Since \(a_d\neq 0\), we have
\(Q\notin I(Z)\), so \(Q\) does not vanish identically on \(Z\). Moreover,
viewing \(a_d\) inside \(\widetilde A_{dm}\), and using
\(\mu\circ \widetilde \Phi=\Phi\), we get
\[
\{\widetilde \Phi^*\widetilde \xi=0\}
\subset
\{\widetilde \Phi^*a_d=0\}
=
\{\Phi^*Q=0\}.
\]
This proves the lemma.
 \end{proof}

\begin{proof}[Proof of Lemma \ref{lem:qs-2}] The proof follows along the lines of the quasi-surjectivity argument from the previous section. In fact, as we shall see, the argument in this case is even simpler since the map is now defined on all of $X$. Since $\tilde Z$ is normal, the singular set consists of finitely many points. Therefore, to prove that $\tilde Z\setminus \tilde\Phi(X)$ is analytic (and hence algebraic), it suffices to prove that $\tilde Z_{reg} \setminus \Phi(X)$ is analytic. We use Simha's criterion once again. Note that since $\tilde\Phi:X\rightarrow \tilde Z_{reg}$ is an injective immersion and $\dim \tilde Z = 2$, it is a biholomorphism, and hence the image $\tilde X = \tilde \Phi(X)$ is Stein. Let $\zeta_0\in \tilde Z_{reg}\setminus\tilde X$ and let $\Delta^2$ be a two-dimensional coordinate polydisc (obtained as before following the arguments after Remark \ref{rem:AfterSimha} after noting that the singular points of $\tilde{Z}$ are finitely many and hence a generic hyperplane will avoid them) in $\tilde Z_{reg}$ centred at $\zeta_0$. Once again to show that the complement is analytic (and hence algebraic) subvariety it suffices to show that the intersection of any hyperplane $H$ with $\Delta^2 \setminus \tilde X$ is finite. As before, let $S_0=H\cap \tilde{Z}$ be a  smooth Riemann surface. It is again enough to show  that the complement of a connected component $\Omega_0$ of $S_0\cap \tilde\Phi(X) $ in $S_0$ is a  finite set. Since $\tilde \Phi$ is a biholomorphism onto its image, we can consider $\Omega_0$ to be a Riemann surface sitting inside $X$. Moreover $\Omega_0$ is a closed subset of $X$ and hence $\omega$ restricts to a complete metric on $\Omega_0$. By Lemma \ref{lem:algebro-geometric} above, $\Omega_0$ is contained in the vanishing set of a section $t\in R_q$ for some $q>>1$, and hence by Proposition \ref{prop:finiteGB} it has finite total Gaussian curvature. An application of  Proposition \ref{prop:huber} completes the proof. The cut-off functions $\xi_a$ needed in Proposition \ref{prop:huber} can simply be taken to be restrictions to $\Omega_0$ of the cut-off functions $\chi_a$ constructed in Lemma \ref{lem:cut-off}.
\end{proof}

\section{Discussion}\label{sec:discussion}

\subsection{Some generalities} In light of the main PDE result from our previous work \cite{DPS},
one can view Yau’s strong conjecture as consisting of three broad
steps. Each step of course is a major open question in its own right.

\begin{itemize}

\item
Let $(X,\omega)$ be a complete non-compact K\"ahler manifold with
$BK>0$. The first step would be to show that $X$ is
\emph{strongly Stein}, in the sense that it admits a smooth
plurisubharmonic exhaustion function
$\rho \colon X \to \RR$ with $|\nabla \rho| \leq 1$.
At present, this appears to be completely open, except of course in the
special case of positive sectional curvature. Wu proved in
\cite{Wu}, using Busemann functions, the existence of a smooth
strictly plurisubharmonic function on such manifolds, but the
properness of the function produced there remains unclear.
One notable feature of Wu’s method is its flexibility: it
associates a strictly plurisubharmonic Busemann-type function to
any family $\{C_t\}_{t>0}$ of closed sets tending to infinity. In the Euclidean
case, this construction recovers the standard linear growth exhaustion
function for instance by considering geodesic spheres.

\item
Assuming a strongly Stein structure, our result in
\cite{DPS} would then produce a uniformly Lipschitz
plurisubharmonic function with finite Monge--Amp\`ere mass.
If, in addition, the B\'ezout-type estimates developed in the
present paper could be extended to higher dimensions, then our
methods should be enough to prove quasi-projectivity, at least modulo
suitable finiteness assumptions on the topology.
\item
The final step would be to pass from quasi-projectivity to the
biholomorphic identification with $\CC^n$. In complex dimension
two, this is closely related to Ramanujam’s theorem. Unfortunately there is no analogue of Ramanujam's theorem in higher dimensions. Indeed, there exist
affine complex threefolds that are contractible and simply
connected at infinity---in fact even diffeomorphic to $\RR^6$---but
 not biholomorphic to $\CC^3$. Thus, even after
quasi-projectivity is established, a substantial rigidity problem
remains. An alternative strategy would be to construct
sufficiently many complete holomorphic vector fields on $X$.
This is the approach taken in \cite{L4}, but the argument there relies crucially on maximal volume growth and the structure theory of {\em non-collapsed} K\"ahler manifolds with lower curvature bounds.
\end{itemize}
\subsection{Pluricomplex Green's functions}
The model singular weight
$\log |z-a|^2$ on $\CC^n$ may be viewed as the basic example of a
pluricomplex Green function. More generally, the use of solutions
to singular Monge--Amp\`ere equations as weights goes back to
Demailly \cite{DMA}, and provides a natural nonlinear replacement
for logarithmic poles in several complex variables.

For bounded domains, the existence theory for such objects is
well developed: in particular, Phong--Sturm \cite{PS} established
the existence of pluricomplex Green functions for compact domains
with boundary. In our setting, one can likewise solve a sequence
of Monge--Amp\`ere equations whose right-hand sides concentrate to
a Dirac mass. The main difficulty is then to extract a useful
limit. At present, obtaining a convergent subsequence with enough
control to serve as a global weight appears to be a rather subtle
problem.

Such a function would be valuable for several reasons. Beyond
providing a canonical singular weight, an exhaustive pluricomplex
Green function would furnish a globally defined analogue of the
Lelong number, opening the door to multiplicity estimates via
monotonicity. From this perspective, the problem is closely tied
to the global complex geometry of the manifold. It is therefore
striking that Stoll’s characterization of $\CC^n$ \cite{Stoll}
is formulated precisely in terms of the existence of an exhaustive
pluricomplex Green function with suitable additional properties.
Related constructions also appear in the work of Griffiths--King
\cite{GriffK}, who produced pluricomplex Green functions on general affine
varieties.

\subsection{Topological aspects}
Very little appears to be known about the topology of complete noncompact K\"ahler manifolds with $BK>0$. For instance, it is unknown whether such manifolds are simply connected. We end with an observation that the B\"ochner formula for harmonic $2$-forms on K\"ahler manifolds with $BK>0$ leads to a restriction on the topology of such manifolds.  Our observation and its proof are similar to a result of Yau \cite{Yau76} for complete Riemannian manifolds with $Ric \ge 0$. Yau proved that the de Rham cohomology group $H^{n-1}(M; \RR)=0$ for such a manifold $M$.

To state our result, recall that $$H^1_{\infty}(M; \RR) := \varinjlim\limits_{K \subset M \text{ compact}} H^1(M \setminus K; \RR).$$
We then have the following observation.
\begin{prop}\label{bochner}

 Let $X$ be a complete noncompact K\"ahler $n$-manifold with $BK\geq 0$. If $H^1_\infty(X;\RR)=0$, then $H^{2n-2}(X;\RR)=0$.

\end{prop}
Note that unlike Yau's aforementioned result for $Ric>0$, the above result is valid only for noncompact manifolds. For instance, if $n=2$,  it implies $H^2(X;\RR)=0$ which is obviously false if $X$ is compact.

\begin{proof} By Poincar\'e Duality, it's enough to show that $H^2_c(X;\RR)=0$. Let $\alpha$ be a real closed $2$-form with compact support. Then by the $L^2$ decomposition theorem there exists a harmonic $L^2$ real $2$-form $\beta$ and a sequence $\{\theta_k\}$ of real smooth, $L^2$ 1-forms such that  $$ \alpha = \beta + \lim_{k\rightarrow\infty}d \ga_k,$$ where the limit is in $L^2$. We first claim that $\be = 0$. We argue as in Yau \cite{Yau76}. We write $\be = \be^{2,0} + \be^{1,1} + \be^{0,2}$ according to type. Note that since $\be$ is real we have $\be^{0,2} = \overline{\be^{2,0}}$ and $\overline{\be^{1,1}} = \be^{1,1}$. Moreover, these forms are all in $L^2$. We will prove separately that $\be^{2,0} = 0$ and $\be^{1,1} = 0$. By a standard B\"ochner formula (cf.\ \cite{BY, GK}), at any $p\in X$, $$\frac{1}{2}\Delta |\be^{1,1}|^2 = |\nabla \be^{1,1}|^ 2+ \sum_{i,j} R_{i\bar i j\bar j}(\la_i-\la_j)^2,$$ where we diagonalize so that $\be^{1,1} = \sum_i \la_i\theta^i\wedge \bar \theta^j,$ where $\{\theta^i\}$ is a unitary frame for $(T_p^*X)^{1,0}$. Note that this step crucially uses the fact that $\be^{1,1}$ is a real form. If $BK\geq 0$, it then follows that $|\be^{1,1}|^2$ is subharmonic. By Cauchy-Schwarz, it also follows that for $\ep>0$, $u_\ep = \sqrt{|\be^{1,1}|^2 + \ep^2} - \ep$ is subharmonic and non-negative. Furthermore, $u_\ep \leq|\be^{1,1}|$ and hence $u_\ep$ is in $L^2$. Then by \cite[Theorem3]{Yau76}, since $Ric \geq  0$, $u_\ep$ is a constant, and hence $|\be^{1,1}|$ is also constant. On the other hand, by a well known result \cite{Calabi,Yau76}, complete noncompact Riemannian manifolds with $Ric \ge 0$ have infinite volume, and hence $\be^{1,1} = 0.$ One can argue similarly to show that $\be^{2,0} = 0$. In fact the B\"ochner formula in this case involves only the Ricci curvature. So we now have that $$\al = \lim_{k\rightarrow \infty}d\ga_k.$$ Then by a classical result of de Rham \cite[Theorem 24]{deRham} (cf. \cite[Lemma 1.11]{Carron} for a proof) it follows that there exists a {\em smooth} $\ga$ such that $\al = d\ga.$

We now finally use the fact that $\al$ is compactly supported. Let $K$ denote the support of $\alpha$, then $\gamma$ is closed on $X \setminus K$. Since $H^1_\infty(X; \RR)=0$, there exists a compact $L$ containing $K$ such that $\gamma$ is exact on $X \setminus L$. Hence there is a smooth function $f: X \setminus L \rightarrow {\mathbb R}$ such that $ df =\gamma \vert_{X \setminus L}$. Extend $f$ smoothly to all of $X$ and let $\gamma_1 = \gamma - df$. Then $\gamma_1$ is compactly supported and $\alpha = d \gamma_1$. It follows that $[\alpha] =0 \in H^2_c(X;\RR)$.
\end{proof}

\begin{rem}
Note that the first part of the argument in fact proves the following statement: If $(X,\omega)$ is a complete non-compact K\"ahler manifold with $BK\geq 0$, there exist no non-trivial real harmonic $2$-forms which are $L^2$ integrable.
\end{rem}

\begin{rem}
In \cite{Yau76}, Yau showed that on a complete Riemannian manifold
$(X,g)$ with $Ric \geq 0$, every closed smooth $L^2$-integrable
$1$-form is exact. Under the stronger assumption $Ric>0$, he
proved moreover that if the form is compactly supported, then one
may choose a compactly supported primitive. The final step relies
on the Cheeger--Gromoll splitting theorem through the fact that
$X$ has only one end, equivalently $H^0_\infty(X,\RR)=\RR$. Since our arguments runs parallel to Yau's argument,  it is natural to
ask the following question in the K\"ahler setting:  Let $(X,\omega)$ be a complete non-compact
K\"ahler manifold with $BK>0$. Must one have
$H^1_\infty(X,\RR)=0$?

If Yau’s conjecture holds, this is immediate, since
$X \simeq \CC^n$. A positive answer, however, would already be
interesting in its own right: it would give an unconditional
vanishing theorem at infinity for complete non-compact K\"ahler
manifolds with positive bisectional curvature, and by the Proposition above, also the vanishing of $H^{2n-2}(X,\RR)$.
\end{rem}

We end with two applications of Proposition \ref{bochner}.

\begin{ex} \normalfont Let $X = \PP^1\times \PP^1 \setminus\Delta$, where $\Delta$ is the diagonal. We claim that $X$ admits a complete K\"ahler metric with $Ric>0$ but admits no such metric with $BK\geq 0$. To prove the first part, write $\bar X = \PP^1\times\PP^1$ and $\sO(a,b):=\pi_1^*\sO_{\PP^1}(a)\otimes \pi_2^*\sO_{\PP^1}(b)$. Since $K_{\bar X}^{-1}\cong \sO(2,2)$ and $\Delta\in |\sO(1,1)|$, we have
$$
\bigl(K_{\bar X}\otimes \sO(\Delta)\bigr)^{-1}\cong \sO(1,1).
$$
Now $\sO(a,b)$ is ample if and only if $a,b>0$, hence $\sO(1,1)$ is ample. Then \cite[Theorem 4.3]{TY} (cf. \cite[Theorem 0.1]{Yeung}) yields a complete K\"ahler metric on $X$ with $Ric>0$.

For the second part, we identify $X$ with a smooth affine quadric. Under the Segre embedding
$$
\sigma\bigl([u_0:u_1],[v_0:v_1]\bigr)=[u_0v_0:u_0v_1:u_1v_0:u_1v_1],
$$
$\PP^1\times\PP^1$ is realized as the quadric $ad-bc=0$ in homogeneous coordinates $[a:b:c:d]$ on $\PP^3$, while $\Delta$ is the hyperplane section $b=c$. Thus on $X$ we may normalize so that $b-c=1$ and set $u=2a$, $v=b+c$, $w=2d$. Substituting $a=u/2$, $b=(v+1)/2$, $c=(v-1)/2$, $d=w/2$ into $ad-bc=0$ gives
$$
Q' := \{(u,v,w)\in \CC^3 : v^2-uw=1\}.
$$
The linear change of coordinates
$$
x=\frac{u-w}{2},\qquad y=v,\qquad z=\frac{i(u+w)}{2}
$$
identifies $Q'$ with
$$
Q_1:=\{(x,y,z)\in \CC^3 : x^2+y^2+z^2=1\}.
$$
Hence $X\cong Q_1$ biholomorphically.

Now $Q_1$ is diffeomorphic to $T^*S^2$, hence also to the total space of $\sO_{\PP^1}(-2)$. In particular, $Q_1$ deformation retracts onto the zero section, so
\begin{equation}\label{eq:bochner-1}
H^2(Q_1,\RR)\cong H^2(S^2,\RR)\cong \RR.
\end{equation}
Moreover, $\sO_{\PP^1}(-2)$ is the minimal resolution of the quadric cone
$$
Q_0:=\{(x,y,z)\in \CC^3 : x^2+y^2+z^2=0\},
$$
and the resolution map restricts to a biholomorphism $\sO_{\PP^1}(-2)\setminus \PP^1\cong Q_0\setminus\{0\}$. It follows that the end of $Q_1$ is diffeomorphic to the punctured cone $Q_0\setminus\{0\}$. Next, $Q_0\setminus\{0\}$ deformation retracts onto its link $Q_0\cap S^5\cong \RR\PP^3$, and hence
\begin{equation}\label{eq:bochner-2}
H^1_\infty(Q_1,\RR)\cong H^1(\RR\PP^3,\RR)=0.
\end{equation}
By \eqref{eq:bochner-1}, \eqref{eq:bochner-2} and Proposition \ref{bochner}, there is no complete K\"ahler metric on $Q_1$, and hence on $X$, with $BK\geq 0$. We conclude by noting that  $Q_1$ is a classical and well-studied example. It carries the complete Ricci-flat K\"ahler metric constructed by Stenzel \cite{Ste}; in complex dimension two this metric is hyperk\"ahler and, after a hyperk\"ahler rotation, is precisely the Eguchi--Hansen metric. In particular, it is asymptotic to the flat K\"ahler cone 
$$
Q_0 \cong \CC^2/\{\pm 1\}=C(\RR\PP^3).
$$\end{ex}
\begin{ex} \normalfont
Next, we consider the Brieskorn surface $$X=F_{2,3,5} =\{(z_1,z_2,z_3) \in \CC^3   :  z_1^2+z_2^3+z_3^5=1\}.$$ By results of Milnor \cite{Mil}, $X$ has the homotopy-type of a wedge of $8$ $2$-spheres, which implies $H_2(X;\RR)={\mathbb R}^8$. Moreover, there  is a compact set $K \subset X$ and a homology $3$-sphere $M$ such that $X \setminus K$ is diffeomorphic to $M \times (0, \infty)$. T In particular, $H^1_\infty(X;\RR)=0$. herefore by Proposition \ref{bochner} above, there exists non-complete K\"ahler metric with $BK \geq 0$. On the other hand, since $X$ has exactly one end and $H_3(X;\RR)=0$, there are no known obstructions to the existence of complete Riemannian metrics with $Ric>0$ on $X$. It appears to be unknown if $X$ actually admits such a metric.
\end{ex}

\section*{Acknowledgements}  The first author (Datar) thanks Jian Song for an interest in the work and useful discussions. The second author (Pingali) thanks Zakarias Sj\"ostr\"om Dyrefelt for his kind hospitality in Aarhus, as well as Tamas Darvas and Claude Lebrun for useful email exchanges. He is also grateful to M.C. Rohan for providing access to ChatGPT pro. We thank ChatGPT for useful feedback and extensive discussions.  The third author (Seshadri) thanks Fangyang Zheng for helpful discussions. The authors also thank Gautam Bharali and Purvi Gupta for discussions, support, and interest in the work.

\end{document}